\documentclass[10pt]{amsart}
\usepackage{amsmath}
\usepackage{amscd}
\usepackage{amssymb}
\usepackage[all]{xy}

\newtheorem{theorem}{Theorem}[section] 
\newtheorem{lemma}[theorem]{Lemma}     
\newtheorem{corollary}[theorem]{Corollary}
\newtheorem{proposition}[theorem]{Proposition}
\newtheorem{definition}[theorem]{Definition}
\newtheorem{remark}[theorem]{Remark}
\newtheorem{example}[theorem]{Example}

\newcommand{\lgarr}{\longrightarrow}

\newcommand{\xarr}{\xrightarrow}

\newcommand{\cat}[1]{\operatorname{\mathsf{#1}}}
\newcommand{\iso}{\stackrel{\sim}{\rightarrow}}

\newcommand{\opn}{\operatorname}

\newcommand{\mcal}[1]{\mathcal{#1}}

\newcommand{\Hom}{\operatorname{Hom}}

\newcommand{\Mod}{\operatorname{\mathsf{Mod}}}

\newcommand{\ten}{\otimes}

\newcommand{\z}{\mathbb{Z}}
\newcommand{\n}{\mathbb{N}}

\newcommand{\qten}{\overset{\bullet q}{\otimes}}
\newcommand{\qhom}{\Hom^{\bullet q}}
\newcommand{\hclim}{\underset{\lgarr}{\cat{hocolim}}}

\newcommand{\ma}{\mathcal{A}}
\newcommand{\mb}{\mathcal{B}}
\newcommand{\md}{\mathbb{D}}

\newcommand{\h}{\operatorname{H}}

\newcommand{\Lqten}{\overset{\mathbb{L} \bullet q}{\otimes}}

\title[NDG categories]
{Derived categories of NDG categories }
\author{Jun-ichi Miyachi and Hiroshi Nagase}
\date{\today}
\address{J. Miyachi: Department of Mathematics, Tokyo Gakugei
University, Koganei-shi, Tokyo, 184-8501, Japan}
\email{miyachi@u-gakugei.ac.jp}
\address{H. Nagase: Department of Mathematics, Tokyo Gakugei
University, Koganei-shi, Tokyo, 184-8501, Japan}
\email{nagase@u-gakugei.ac.jp}

\begin{document}

\maketitle

\tableofcontents

\begin{abstract}
We study $N$-differential graded ($NDG$) categories and their the derived  categories.
First, we introduce $N$-differential modules over an $NDG$ category $\ma$.
Then we show that the category $\cat{C}_{Ndg}(\ma)$ of $N$-differential $\ma$-modules
is a Frobenius category, and that its homotopy category $\cat{K}_{Ndg}(\ma)$ is a triangulated
category.
Second, we study the properties of the derived category $\cat{D}_{Ndg}(\ma)$ and 
give triangle equivalences of Morita type between derived categories of $NDG$ categories.
Finally, we show $\cat{D}_{Ndg}(\ma)$ is  triangle equivalent to the derived category of some ordinary DG category.
\end{abstract}


\setcounter{section}{-1}
\section{Preliminaries}\label{intro}

The notion of $N$-complexes, that is, graded objects with $N$-differential $d$
($d^{N}=0$), was introduced by Kapranov (\cite{Ka}).
He also described the relation between an $N$-differential and a primitive $N$-th root $q$ of unity.
Dubois-Violette studied homological properties of $N$-complexes (\cite{D}). 
Afterwards, Iyama, Kato and Miyachi studied derived categories of $N$-complexes of an abelian category (\cite{IKM3}).
Dubois-Violette and Kerner introduced $N$-differential $\n$-graded algebras, and studied homological properties (\cite{D2}, \cite{DK}). 
In this article, we study the derived  category of an $N$-differential $\z$-graded ($NDG$) category $\ma$.
In particular we study the homotopy category $\cat{K}_{Ndg}(\ma)$ of right $NDG$ $\ma$-modules
and the derived category $\cat{D}_{Ndg}(\ma)$.

Let $k$ be a commutative ring, $q$ a primitive $N$-th root of unity of $k$, $U, V$ $N$-differential complex of $k$-modules.
Let 
$U\qten V$ is the tensor product
$(\coprod_{n \in \z}U\overset{n}{\ten}V, d_{U\qten V}(u\otimes v))$
where
$U\overset{n}{\ten}V=\coprod_{r+s=n}U^r\otimes V^s$
with
an $N$-differential
$
d_{U\qten V}(u\otimes v)=d_U(u)\otimes v+q^r u\otimes d_V(v)
$
for any $u \in U^r$ and $v\in V$.
Then $U\qten V$ is also an $N$-complex (\cite{Ka}).
An $N$-complex $V\qten U$ is also defined.
But $V\qten U$ is not isomorphic to $U\qten V$ in case of $N>2$ (Lemma \ref{qten01}).
Moreover it is difficult to find the tensor product of $NDG$-algebras (\cite{Si}).
This situation implies the difficulty to deal with $NDG$ modules, especially $NDG$ bi-modules as functors from $NDG$ categories.
Therefore we deal with $NDG$ modules using module-theoretical language (Definitions \ref{def:NDGmod}, \ref{NDG-bimodule}).

In Section \ref{qnumber} we give some formulas on $q$-numbers.
We use these formulas to describe $NDG$ modules with complicated scalar multiplication of $NDG$ categories.

In Section \ref{NDG}, we define an $N$-differential graded category $\ma$ with respect to $q$ (=$N_qDG$ category), and 
a right $NDG$ $\ma$-module, that is a graded module with right scalar multiplication of $\ma$.
We show the following two results are obtained similarly to the case of $N$-complexes (\cite{IKM3}).
First, we show that the category  $\cat{C}_{Ndg}(\ma)$ of right $NDG$ $\ma$-modules is a Frobenius category.

\begin{theorem}[Theorem \ref{Kalgtricat}]
The homotopy category $\cat{K}_{Ndg}(\ma)$ is an algebraic triangulated category.
\end{theorem}
Moreover, $\cat{C}_{Ndg}(\ma)$ has two auto-equivalences: the shift functor $\theta_q$ of grading and the suspension functor $\Sigma$ 
as a triangulated category. Then we have the following relation.

\begin{theorem}[Theorem \ref{Sigmatheta}]
\[
\Sigma ^2 \simeq \theta^N_q ~\mbox{in}~ \cat{K}_{Ndg}(\ma).
\]
\end{theorem}

Furthermore, we define an $NDG$ $\mb$-$\ma$-bimodule $M$, and we have triangle functors 
\[\begin{aligned}
\qhom_{\ma}(M, -): \cat{K}_{Ndg}(\ma) \to \cat{K}_{Ndg}(\mb), &-\qten_{\mb}M: \cat{K}_{Ndg}(\mb) \to \cat{K}_{Ndg}(\ma) .
\end{aligned}\]

\begin{theorem}[Theorems \ref{adjNDG}]
For an $NDG$ $\mb$-$\ma$-bimodule $M$, 
$-\qten_{\mb} M$ is the left  adjoint of $\qhom_{\ma}(M, -)$.
\end{theorem}

In section \ref{DGNDG}, 
we describe the derived categories of $NDG$ modules, and the derived functors induced by bimodules,
especially the necessary and sufficient conditions for these functors to be  equivalences.

\begin{theorem}[Theorem \ref{projinj}] 
For an $N_qDG$ category $\ma$, we have the following triangle equivalent
\[
\cat{D}_{Ndg}(\ma) \simeq \cat{K}_{Ndg}^{\mathbb{P}}(\ma) \simeq \cat{K}_{Ndg}^{\mathbb{I}}(\ma).
\]
Here $\cat{K}_{Ndg}^{\mathbb{P}}(\ma)$ (resp., $\cat{K}_{Ndg}^{\mathbb{I}}(\ma)$) 
is the smallest full triangulated subcategory of $\cat{K}_{Ndg}(\ma)$ consisting of
$\theta_q^nA\,\hat{}$ (resp. $ \theta^n_q \mathbb{D}~\hat{}A$) for all $n \in \z$ and $A \in \ma$
which is closed under taking coproducts (resp. products), where $A\,\hat{}=\ma(-,A),\ \hat{}A=\ma(A,-)$, 
$\md \hat{}A=\qhom _k(\hat{}A, E)$ for an injective cogenerator $E$ for $k$-modules.
\end{theorem}

\begin{theorem}[Theorem \ref{Morita}]
Let $\ma$ and $\mb$ be small $N_qDG$ categories. 
The following hold for an $NDG$ $\mb$-$\ma$-bimodule $M$.
\begin{enumerate}
\item The functor $-\Lqten_{\mb}M : \cat{D} _{Ndg}(\mb) \to \cat{D} _{Ndg}(\ma)$
is the left adjoint of the functor $\mathbb{R}\Hom_{\ma}^{\bullet q}(M, -)$.
\item The following are equivalent.
\begin{enumerate}
\item $-\Lqten_{\mb}M$ is a triangle equivalence.
\item $\mathbb{R}\Hom_{\ma}^{\bullet q}(M, -)$ is a triangle equivalence.
\item $\{ \Sigma ^i \theta^{-n}_q B\,\hat{}\otimes _{\mb}^{\bullet q}M\; |\; B\in \mb, 0 \leq n \leq N-1, i \in \z \}$
is a set of compact generators for $\cat{D} _{Ndg}(\ma)$, and
the canonical morphism 
\[
\Hom_{\cat{D}(\mb)}(\theta^{j}B_1\hat{}, \Sigma^{i} B_2\hat{}) \iso 
\Hom_{\cat{D}_{Ndg}(\ma)}(\theta^{j}B_1\hat{} \qten_{\mb}M, \Sigma^{i} B_2\hat{} \qten_{\mb}M)
\]
is an isomorphism for $i=0,1$, any $j$ and any $B_1, B_2 \in \mb$.
\end{enumerate}
If $k$ is a field, then the above  are equivalent to
\begin{enumerate}
\item[(d)] $\{ \Sigma ^i \theta^{-n}_q B\,\hat{}\otimes _{\mb}^{\bullet q}M\; |\; B\in \mb, 0 \leq n \leq N-1, i \in \z \}$
is a set of compact generators for $\cat{D} _{Ndg}(\ma)$, 
the canonical map 
\[
\mb(B_1, B_2) \to 
\mathbb{R}\Hom_{\ma}^{\bullet q}(B_1\hat{} \qten_{\mb}M, B_2\hat{} \qten_{\mb}M)
\]
is an isomorphism in $\cat{D}_{Ndg}(k)$ for any $B_1, B_2 \in \mb$.
\end{enumerate}
\end{enumerate}
\end{theorem}

Finally, we show that every derived category of $NDG$ modules is triangle equivalent to derived categories of some ordinary $DG$ category.

\begin{theorem}[Theorem \ref{NdgDG}]
For any $N_qDG$ category $\ma$, there exists a $DG$ category $\mb$ such that
$\cat{D} _{Ndg}(\ma)$ is triangle equivalent to the derived category $\cat{D} _{dg}(\mb)$.

\end{theorem}

Throughout this note we fix a positive integer $N \ge 2$ and a commutative ring $k$ with unity $1$ 
and we assume that $k$ has a primitive $N$-th root  $q$ of $1$.
We will write $\Hom$ and $\otimes$ for $\Hom_k$ and $\ten_k$.


\section{$q$-numbers}\label{qnumber}
In this section, we collect some formulas on $q$-numbers. We will omit proofs of some elementary facts.
For a commutative ring $k$ with unity $1$ and a positive integer $N \ge 2$,
we say that an element $q$ of $k$ is a {\it primitive $N$-th root} of $1$ if
$q^N-1=0$ and $q^l-1$ is a non-zero-divisor for any $1 \le l \le N-1$.
Then $[m]=1+ q + \cdots + q^{m-1}$ is a non-zero-divisor for $1 \leq m <N$, and $[N]=0$.
We define $q$-numbers in the total quotient ring $K$ of $k$.
For any positive integer $m$,
$[m]!=[m][m-1]\cdots[1]$ and $[0]!=1$.
For $0 \le l \le m \le N$ with  $(l, m)\ne(N, N)$ nor $(0, N)$,
\[
\begin{bmatrix} m \\ l \\ \end{bmatrix}
=\frac{[m]!}{[l]![m-l]!}
\]
and
\[
\begin{bmatrix} N \\ N \\ \end{bmatrix}
=
\begin{bmatrix} N \\ 0 \\ \end{bmatrix}
=1.
\]
Then we have 
$
\begin{bmatrix} m \\ l \\ \end{bmatrix}=
\begin{bmatrix} m \\ m-l \\ \end{bmatrix}
$.

By the following lemma, we can consider 
$\begin{bmatrix}
m \\
l \\
\end{bmatrix}$ lies in $k$
through the canonical injection $k \to K$.

\begin{lemma}\label{qn}
For $1 \le l < m \le N$,
\[
\begin{bmatrix} m-1 \\ l-1 \\ \end{bmatrix} 
+\begin{bmatrix} m-1 \\ l \\ \end{bmatrix} q^{l}
=\begin{bmatrix} m \\ l \\ \end{bmatrix}, 
\]
\[
\begin{bmatrix} m-1 \\ l-1 \\ \end{bmatrix} q^{m-l}
+\begin{bmatrix} m-1 \\ l \\ \end{bmatrix} 
=\begin{bmatrix} m \\ l \\ \end{bmatrix}.
\]
\end{lemma}

We show the following lemma for the lemma after the next. 

\begin{lemma}\label{qn0}
For $1 \le t \le N$.
\[
\sum^{t}_{j=0}(-1)^{j}q^{\frac{j(j-1)}{2}}\begin{bmatrix} t \\ j \\ \end{bmatrix}=0.
\]
\end{lemma}

\begin{proof}
Set $a_t=\sum^{t}_{j=0}(-1)^{j}q^{\frac{j(j-1)}{2}}\begin{bmatrix} t \\ j \\ \end{bmatrix}$.
We have that $a_1=1-1=0$. 
For any $2 \le t \le N$, by Lemma \ref{qn},
\begin{align*}
a_t&=\sum^{t-1}_{j=0}(-1)^{j}q^{\frac{j(j-1)}{2}}\begin{bmatrix} t-1 \\ j \\ \end{bmatrix}
+\sum^{t}_{j=1}(-1)^{j}q^{\frac{j(j-1)}{2}}\begin{bmatrix} t-1 \\ j-1 \\ \end{bmatrix}q^{t-j}\\
&=a_{t-1}-q^{t-1}a_{t-1}.
\end{align*}
Therefore $a_t=0$ for any $1\leq t \le N$.
\end{proof}

\begin{lemma}\label{qnnp}
Let $n$ be an integer with $1 \le n \le N$ and $X$ a $k$-module. 
For any morphisms $\phi$, $\psi \in \Hom(X, X)$ with $\psi\phi = q\phi\psi$, the following hold.
\begin{enumerate}
\item $(\phi +\psi)^n= \sum_{l=0}^n \begin{bmatrix} n \\ l \\ \end{bmatrix}\phi^{n-l}\psi^l$,
\item $\phi^n= \sum_{l=0}^n (-1)^lq^{\frac{l(l-1)}{2}}\begin{bmatrix} n \\ l \\ \end{bmatrix}(\phi+\psi)^{n-l}\psi^l.$
\end{enumerate}
\end{lemma}

\begin{proof}
Set $a_n=\sum_{l=0}^n \begin{bmatrix} n \\ l \\ \end{bmatrix}\phi^{n-l}\psi^l$.
By Lemma \ref{qn},
\begin{align*}
a_n=&\sum_{l=0}^{n-1} \begin{bmatrix} n-1 \\ l \\ \end{bmatrix}\phi^{n-l}\psi^l+
\sum_{l=1}^{n} \begin{bmatrix} n-1 \\ l-1 \\ \end{bmatrix}q^{n-l}\phi^{n-l}\psi^l\\
=&\phi\sum_{l=0}^{n-1} \begin{bmatrix} n-1 \\ l \\ \end{bmatrix}\phi^{n-1-l}\psi^l+
\psi\sum_{l=1}^{n} \begin{bmatrix} n-1 \\ l-1 \\ \end{bmatrix}\phi^{n-l}\psi^{l-1}\\
=&(\phi+\psi)a_{n-1}\\
=&(\phi+\psi)^n .
\end{align*}
Set $b_n=\sum_{l=0}^n (-1)^lq^{\frac{l(l-1)}{2}}\begin{bmatrix} n \\ l \\ \end{bmatrix}(\phi+\psi)^{n-l}\psi^l$.
By Lemma \ref{qnnp} (1),
\begin{align*}
b_n=&\sum_{l=0}^n (-1)^lq^{\frac{l(l-1)}{2}}\begin{bmatrix} n \\ l \\ \end{bmatrix}
\sum_{j=0}^{n-l} \begin{bmatrix} n-l \\ j \\ \end{bmatrix}\phi^{n-l-j}\psi^j\psi^l\\
=&\sum_{l=0}^n\sum_{j=0}^{n-l}(-1)^lq^{\frac{l(l-1)}{2}}\begin{bmatrix} n \\ l \\ \end{bmatrix}
\begin{bmatrix} n-l \\ j \\ \end{bmatrix}\phi^{n-(l+j)}\psi^{l+j}\\
=&\sum_{l=0}^n\sum_{j=0}^{n-l}(-1)^lq^{\frac{l(l-1)}{2}}\begin{bmatrix} l+j \\ l \\ \end{bmatrix}
\begin{bmatrix} n \\ l+j \\ \end{bmatrix}\phi^{n-(l+j)}\psi^{l+j}\\
=&\sum_{t=0}^n(\sum_{l=0}^{t}(-1)^lq^{\frac{l(l-1)}{2}}\begin{bmatrix} t \\ l \\ \end{bmatrix})
\begin{bmatrix} n \\ t \\ \end{bmatrix}\phi^{n-t}\psi^{t}.
\end{align*}
By Lemma \ref{qn0}, $b_n=\phi^n$.
\end{proof}

We prepare the following lemma also, which will be used later.

\begin{lemma}\label{qninv}
For any $1\le t \le N-1$ and $0 \le l \le N-t$,
\[
(-1)^lq^{lt+\frac{l(l-1)}{2}}\begin{bmatrix} N-t \\ l\\ \end{bmatrix} 
=\begin{bmatrix} l+t-1 \\ l \\ \end{bmatrix} .
\]
In particular,
$
(-1)^{l}q^{\frac{l(l+1)}{2}} \begin{bmatrix} N-1 \\ l \\ \end{bmatrix}=1
$
for $0 \le l \le N-1$.
\end{lemma}

\begin{proof}
Since $[j]+q^{j}[N-j]=[N]=0$ for any $1 \le j \le N-1$,
\[
[t][t+1]\cdots[t+l-1]=(-1)^lq^{lt+\frac{j(j+1)}{2}}[N-t][N-(t+1)]\cdots[N-(t+l-1)]
\]
for $l \ge 1$. Hence
\[
(-1)^lq^{lt+\frac{l(l-1)}{2}}\begin{bmatrix} N-t \\ l\\ \end{bmatrix} 
=\begin{bmatrix} l+t-1 \\ l \\ \end{bmatrix}.
\]
\end{proof}


\section{NDG categories}\label{NDG}

In this section, we introduce $N$-differential graded modules over an $N$-differential graded category and
give the definition of some categories.

We recall the definition of the category of $\z$-graded $k$-modules.
The category  $\cat{Gr}(k)$ of $\mathbb{Z}$-graded $k$-modules is defined as follows:
\begin{enumerate}
\item An object is a $\mathbb{Z}$-graded $k$-module $U=\coprod_{i \in \mathbb{Z}}U^i$.
\item The morphism set between $U$ and $V$ is given by 
\[
\Hom_{\cat{Gr}(k)}(U, V):=\coprod_{i \in \mathbb{Z}}\Hom^i(U, V)
\]
where
\[
\Hom^i(U, V):=\prod_{l \in \z}\Hom(U^l,V^{l+i}).
\]
The composition
\[
\Hom_{\cat{Gr}(k)}(V, W) \otimes \Hom_{\cat{Gr}(k)}(U, V) \to  \Hom_{\cat{Gr}(k)}(U, W), \; f \otimes g \mapsto fg 
\]
is $k$-bilinear and homogeneous of degree $0$.
\end{enumerate}

The category $\cat{Gr}^0(k)$ is the subcategory of $\cat{Gr}(k)$ whose objects are same as $\cat{Gr}(k)$ and the morphism set
between $X$ and $Y$ is given by
\[
\Hom_{\cat{Gr}^0(k)}=\Hom^0(X, Y).
\]


\subsection{$N$-differential graded $k$-modules}\label{NDGmod}

In this subsection we study $N$-differential graded ($NDG$) $k$-modules.

A sequence $X=(X, d_X)$ of $k$-modules is called an $N$-differential graded (NDG) $k$-module 
if a $\z$-graded $k$-module $X$ endowed with $d_X \in \Hom_{\cat{Gr}(k)}^1(X, X)$ satisfying $d_X^{\{N\}}=0$.
Here $d^{\{i\}}$ means the $i$-th power of $d$.
The category $\cat{C}_{Ndg}(k)$ of $N$-differential $k$-modules is defined as follows:
\begin{enumerate}
\item Objects are $N$-differential graded $k$-modules.
\item The morphism set between $X=(X, d_X)$ and $Y=(Y,d_Y)$ is given by
\[
\Hom_{\cat{C}_{Ndg}(k)}(X, Y)=\{ f \in \Hom^0(X, Y) \mid f\circ d_X=d_Y\circ f\}
\]
and the composition is given by the composition of maps.
\end{enumerate}

\begin{remark}\label{AbNdgk}
It is easy to see that $\cat{C}_{Ndg}(k)$ and $\cat{Gr}^0(k)$ are abelian categories with ordinary exact sequences.
\end{remark}

For NDG  $k$-modules $U$ and $V$,
a sequence $\qhom (U, V)$ is defined as follows.
\[
\qhom (U, V):=\Hom_{\cat{Gr}(k)}(U, V)
\]
and
\[
d_{\qhom (U, V)}(f)=d_V\circ f-q^rf\circ d_U
\]
for any $f \in \Hom^r(U, V)$. 
A sequence $U\qten V$ is defined as follows.
\[
U\qten V=\coprod_{i \in \z}U\otimes^iV
\]
where
\[
U\otimes^iV=\coprod_{l+m=i}U^l\otimes V^m
\]
for any $i \in \z$, and
\[
d_{U\qten V}(u\otimes v)=d_U(u)\otimes v+q^su\otimes d_V(v)
\]
for any $u \in U^s$ and $v\in V$.

Then we have the following lemma.

\begin{lemma}\label{homtennp}
For any $n \in \n$ with $1\le n \le N$ and sequences of  $k$-modules $U$ and $V$, the following hold.
\begin{enumerate}
\item For any $f \in \Hom^r(U, V)$,
\[
d_{\qhom (U, V)}^{\{n\}}(f)=\sum ^{n}_{l=0}(-1)^lq^{lr+\frac{l(l-1)}{2}}
\begin{bmatrix} n \\ l \\ \end{bmatrix}d^{\{n-l\}}_V\circ f\circ d^{\{l\}}_U.
\]
\item For any $u \in U^s$ and $v \in V$,
\[
d_{U\qten V}^{\{n\}}(u\otimes v)=\sum ^n_{l=0}q^{ls}
\begin{bmatrix} n \\ l \\ \end{bmatrix}d^{\{n-l\}}_U(u)\otimes d^{\{l\}}_V(v).
\]
\end{enumerate}
\end{lemma}

\begin{proof}
For any $f \in \Hom^r(U, V)$,
set $\phi(f)=d_V\circ f$ and $\psi(f)=-q^rf\circ d_U$, then
$d_{\qhom (U, V)}(f)=(\phi+\psi)(f)$ and
$\psi\phi(f)=q\phi\psi(f)$. By Lemma \ref{qnnp},
\begin{align*}
d_{\qhom (U, V)}^{\{n\}}(f)=&(\phi+\psi)^n(f)\\
=& \sum_{l=0}^n \begin{bmatrix} n \\ l \\ \end{bmatrix}\phi^{n-l}\psi^l(f)\\
=&\sum ^{n}_{l=0}(-1)^lq^{lr+\frac{l(l-1)}{2}}
\begin{bmatrix} n \\ l \\ \end{bmatrix}d^{\{n-l\}}_V\circ f\circ d^{\{l\}}_U.
\end{align*}
For any $u \in U^s$, $v \in V$, set $\phi(u\otimes v)=d_U(u)\otimes v$ and $\psi(u\otimes v)=q^su\otimes d_V(v)$,
then $d_{U\qten V}(u\otimes v)=(\phi+\psi)(u\otimes v)$ and 
$\psi\phi(u\otimes v)=q\phi\psi(u\otimes v)$. By Lemma \ref{qnnp},
\begin{align*}
d_{U\qten V}^{\{n\}}(u\otimes v)=&(\phi+\psi)^n(u\otimes v)\\
=& \sum_{l=0}^n \begin{bmatrix} n \\ l \\ \end{bmatrix}\phi^{n-l}\psi^l(u\otimes v)\\
=&\sum ^n_{l=0}q^{ls}
\begin{bmatrix} n \\ l \\ \end{bmatrix}d^{\{n-l\}}_U(u)\otimes d^{\{l\}}_V(v).
\end{align*}
\end{proof}

\begin{corollary}[\cite{Ka}]\label{homtenNdg}
For NDG $k$-modules $U$ and $V$,
the sequences $\qhom (U, V)$ and $U\qten V$ are $N$-differential $k$-modules.
\end{corollary}

\begin{proof}
Since
$0=[N]!=[l]![N-l]!\begin{bmatrix} N \\ l \\ \end{bmatrix}$
and $[l]![N-l]!$ is  a non-zero-divisor for any $1 \le l \le N-1$,
we have $\begin{bmatrix} N \\ l \\ \end{bmatrix}=0$ for any $1 \le l \le N-1$.
By Lemma \ref{homtennp},
\[
d_{\qhom (U, V)}^{\{N\}}(f)=d^{\{N\}}_V\circ f+(-1)^{N}q^{\frac{N(N-1)}{2}}f\circ d^{\{N\}}_U=0
\]
for any $f \in \Hom(U, V)$.
Similarly we have
$d_{U\qten V}^{\{N\}}=0$.
\end{proof}
 
 The following lemma is easy to check.
 
\begin{lemma}\label{qten01}
For $NDG$ $k$-modules $U, V$ and $W$, the following hold.
\begin{enumerate}
\item We have the canonical isomorphism $(U\qten V)\qten W \simeq U\qten (V\qten W)$ in $\cat{C}_{Ndg}(k)$.
\item $V$ induces the functors $V\qten -: \cat{C}_{Ndg}(k) \to \cat{C}_{Ndg}(k)$, $- \qten V: \cat{C}_{Ndg}(k) \to \cat{C}_{Ndg}(k)$.
\item $V$ induces the functors $\qhom(V, -) : \cat{C}_{Ndg}(k) \to \cat{C}_{Ndg}(k)$, $\qhom(-, V): \cat{C}_{Ndg}(k) \to \cat{C}_{Ndg}(k)$.
\item We have the canonical isomorphism 
\[
\Hom_{\cat{C}_{Ndg}(k)}(U\qten V,W) \simeq \Hom_{\cat{C}_{Ndg}(k)}(U, \qhom(V, W)).
\]
\item We have an isomorphism $U\qten V \simeq V\overset{\bullet}{\ten}^{q^{-1}} \!\! U$ in $\cat{C}_{Ndg}(k)$.
\item We have an isomorphism 
\[\Hom_{\cat{C}_{Ndg}(k)}(U\qten V,W) \simeq \Hom_{\cat{C}_{Ndg}(k)}(V, \Hom^{\bullet q^{-1}} (U, W)).
\]
\end{enumerate}
\end{lemma}

\begin{proof}
(1), (2), (3) It is easy.
\par\noindent
(4) By the ordinary isomorphism $\Hom(U^r\ten V^s, W^{r+s}) \iso \Hom(U^r \Hom(V^s, W^{r+s}))$, it is easy to check.
\par\noindent
(5) The morphism $U\qten V \to V\overset{\bullet}{\ten}^{q^{-1}}\!\! U$ which is defined by $U^r\ten V^s \to V^s\ten U^r ~(u\ten v \mapsto q^{rs}v \ten u)$ is an isomorphism in $\cat{C}_{Ndg}(k)$.
\par\noindent
(6) By (4) and (5).
\end{proof}

\begin{remark}\label{exhomtenNdgk}
In the above $V\qten-$ (resp., $\qhom(V,-)$) is left (resp., right) exact in the sense of Remark \ref{AbNdgk}.
\end{remark}


\subsection{$N$-differential graded categories}\label{NDGcat}

In this subsection, we give the definition of 
the $N$-differential graded category with respect to a primitive $N$-th root $q$ of $1$ and 
the definition of the homotopy category of the $N$-differential graded category.

\begin{definition}\label{def:NDG}
An $N_qDG$ category $\ma$ is defined by the following datum.
\begin{enumerate}
\item A class of objects $Ob\ma$.
\item The morphism set
\[
\ma(A, B)=\coprod_{i \in \z}\ma^i(A, B)
\]
which is an NDG $k$-module for any $A$ and $B$ in $Ob\ma$, and
the composition
\[\begin{array}{rll}
\mu : & \ma(B,C)\qten \ma(A,B) \to \ma(A,C), & \mbox{in}~\cat{C}_{Ndg}(k). \\
 & (f \otimes g \mapsto fg)
\end{array}\]
That is, $\mu \circ(\mu\ten 1)=\mu \circ(1\ten\mu)$ in $\cat{C}_{Ndg}(k)$.
\end{enumerate}
\end{definition}

The above condition that $\mu$ lies in $\cat{C}_{Ndg}(k)$ is equivalent to
that $\mu$ lies in $\cat{Gr}^0(k)$ and
satisfies
$d_{\ma(A,C)}\circ \mu=\mu \circ d_{\ma(B,C)\qten \ma(A,B)},$
namely,
\[
d(fg)=d(f)g+q^rfd(g)
\]
for any $f \in \ma ^r(B,C)$ and $g \in \ma(A,B)$.

\begin{example}\label{NqDG(k)}
Let $N_qdg(k)$ be the category of $NDG$ $k$-modules of which morphism sets $\qhom(X, Y)$ for all $NDG$ $k$-modules $X, Y$.
Then $N_qdg(k)$ is a $N_qDG$ category.
\end{example}

We give the definitions of right $N$- differential graded modules over an $N_qDG$ category and
morphisms between these modules.

\begin{definition}\label{def:Grmod}
Let $\ma$ be a $N_qDG$ category. A right graded $\ma$-module $X=(X ; \rho_X)$ is a collection of 
a graded $k$-module $X(A)$ for any objet $A$ in $\ma$ with a homogeneous morphism of degree $0$ as a scalar multiplication
\[\begin{array}{rll}
\rho_X :  & X(A_1)\otimes \ma (A_2, A_1) \to X(A_2)  & \text{in $\cat{Gr}^0(k)$} \\
& ( x\otimes f \mapsto xf)
\end{array}\]
satisfying 
\begin{enumerate}
\item $x1_{A}=x$ for any $A \in \ma$, $x \in X(A)$ and $1_{A} \in \ma (A, A)$.
\item $\rho_X\circ(1\ten\mu)=\rho_X\circ(\rho_X\ten 1)$ in $\cat{Gr}^0(k)$.
\end{enumerate}
Namely, we have $x(fg)=(xf)g$ for any $x \in X(A_1)$, $f \in \ma (A_2, A_1)$, $g\in \ma (A_3, A_2)$ and any $A_1, A_2, A_3 \in \ma$. 
\par\noindent
The category $\cat{Gr}^0(\ma)$ of right graded $\ma$-modules is defined by the following datum:
\begin{enumerate}
\item objects are right graded $\ma$-modules.
\item For right graded $\ma$-modules$X=(X; \rho_X), Y=(Y; \rho_Y)$, a morphism $F:X \to Y$ is defined by a collection of morphisms $F_A: X(A) \to Y(A)$ in $\cat{Gr}^0(k)$
such that 
\[\begin{CD}
X(A_1)\otimes \ma(A_2, A_1) @>\rho_X>> X(A_2)  \\
@VF_{A_1}\ten 1VV @VVF_{A_2}V \\
Y(A_1)\otimes \ma(A_2, A_1) @> \rho_Y>> Y(A_2) 
\end{CD}\]
is commutative in $\cat{Gr}^0(k)$.
\end{enumerate}
\end{definition}

\begin{definition}\label{def:NDGmod}
A right NDG $\ma$-module $X=(X, d_X; \rho_X)$ is a collection of 
an NDG $k$-module $X(A)=(X(A), d_{X(A)})$ for any objet $A$ of $\ma$
with a homogeneous morphism of degree $0$ as a scalar multiplication
\[\begin{array}{rll}
\rho_X : X(A_1)\qten  \ma (A_2, A_1) \to X(A_2) & \text{in $\cat{C}_{Ndg}(k)$}
\end{array}\]
satisfying 
\begin{enumerate}
\item $x1_{A}=x$ for any $A \in \ma$, $x \in X(A)$ and $1_{A} \in \ma (A, A)$.
\item $\rho_X\circ(1\ten\mu)=\rho_X\circ(\rho_X\ten 1)$ in $\cat{C}_{Ndg}(k)$.
\end{enumerate}
The category $\cat{C}_{Ndg}(\ma)$ of right NDG $\ma$-modules is defined by the following datum:
\begin{enumerate}
\item objects are right $NDG$ $\ma$-modules.
\item For right $NDG$ $\ma$-modules$X=(X, \rho_X), Y=(Y, \rho_Y)$, a morphism $F:X \to Y$ is defined by a collection of morphisms $F_A: X(A) \to Y(A)$ in $\cat{C}_{Ndg}(k)$
such that 
\[\begin{CD}
X(A_1)\qten \ma(A_2, A_1) @>\rho_X>> X(A_2)  \\
@VF_{A_1}\ten 1VV @VVF_{A_2}V \\
Y(A_1)\qten \ma(A_2, A_1) @> \rho_Y>> Y(A_2) 
\end{CD}\]
is commutative in $\cat{C}_{Ndg}(k)$.
\end{enumerate}
\par\noindent
Similarly, left graded $\ma$-modules and left $NDG$ $\ma$-modules are defined.
\end{definition}

\begin{remark}\label{bimod}
As well as Remark \ref{AbNdgk}, it is easy to see that $\cat{C}_{Ndg}(\ma)$ and $\cat{Gr}^0(\ma)$ are abelian categories with ordinary exact sequences.
\par\noindent
By Lemma \ref{qten01}, we know that the above definition of left $NDG$ $\ma$-modules is equivalent to covariant functors from $\ma$ to $N_qdg(k)$ as the same as \cite{Ke1}.
But by the language of functors, right $NDG$ $\ma$-modules are defined by contravariant functors from $\ma$ to the $N_{q^{-1}}DG$ category $N_{q^{-1}}dg(k)$,
because by Lemma \ref{qten01} the opposite category $\ma^{op}$ of an $N_qDG$ category $\ma$ is an $N_{q^{-1}}DG$ category.
Therefore in case of $N >2$ it is difficult to define right NDG $\ma$-modules, especially NDG $\mb$-$\ma$-bimodules on the same ground (see Definition \ref{NDG-bimodule}).
\end{remark}

\begin{example}
For any $N_qDG$ category $\ma$ and any object $A$ of $\ma$,
$A\,\hat{}=\ma(-, A)$ is a right NDG $\ma$-module and
$\;\hat{}A=\ma(A, -)$ is a left NDG $\ma$-module.
\end{example}

Let $\ma$ be an $N_qDG$ category. 
A homogeneous morphism $F : X \to Y$ of degree $n$ between right graded $\ma$-modules $X$ and $Y$ is
a collection of $F_A$ in $\Hom^n(X(A), Y(A))$ for any object $A$ of $\ma$ satisfying 
\[
F_{A_2}(xf)=F_{A_1}(x)f
\] 
for any $x \in X(A_1)$ and $f \in \ma(A_2, A_1)$.

We denote by $\Hom_{\ma}^n(X, Y)$ 
the set of  homogeneous morphisms of degree $n$ between right graded $\ma$-modules $X$ and $Y$.
For $F \in \Hom_{\ma}^n(X, Y)$ and $G \in \Hom_{\ma}^m(Y, Z)$,
we define $(GF)_A=G_AF_A$, so $GF \in \Hom_{\ma}^{n+m}(X, Z)$.

\begin{definition}
Let $\ma$ be an $N_qDG$ category. 
The category $\cat{Gr}(\ma)$ of right graded $\ma$-modules is defined as follows.
\begin{enumerate}
\item Objects are right graded $\ma$-modules.
\item The morphism set between $X$ and $Y$ is given by
\[
\Hom_{\cat{Gr}(\ma)}(X, Y)=\coprod_{i \in \mathbb{Z}}\Hom_{\ma}^i(X, Y).
\]
\end{enumerate}
\end{definition}

To give the definition of 
the homotopy category $\cat{K}_{Ndg}(\ma)$ over an $N_qDG$ category $\ma$,
we show the following lemmas.

\begin{lemma}\label{NGD(A)}
Let $\ma$ be an $N_qDG$ category, $X$, $Y$ right NDG $\ma$-modules.
For any $F\in \Hom_{\ma}^{i}(X, Y)$ and $i \in \z$,
we have 
\[
d_Y\circ F-q^iF\circ d_X \in \Hom_{\ma}^{i+1}(X, Y) .
\]
Therefore, $(\coprod_{i\in \z}\Hom_{\ma}^i(X, Y), (d_Y\circ (-) -q^i(-)\circ d_X)_{i \in\z})$
is an $N$-differential graded $k$-module.
We denote  $(\coprod_{i\in \z}\Hom_{\ma}^i(X, Y), (d_Y\circ (-) -q^i(-)\circ d_X)_{i \in\z})$
by 
\[\qhom_{\ma}(X, Y)=(\qhom_{\ma}(X, Y), d_{\qhom_{\ma}(X,Y)}) .
\]
\end{lemma}

\begin{proof}
Given $F \in \Hom_{\ma}^i(X, Y)$,
for any $x \in X^s(A_1)$ and $f \in \ma(A_2, A_1)$,
\begin{align*}
(d_{\qhom_{\ma} (X, Y)}(F))_{A_2}(xf)
=&(d\circ F_{A_2}-q^iF_{A_2}\circ d)(xf)\\
=&d(F_{A_1}(x)f)-q^iF_{A_2}(d(x)f+q^sxd(f))\\
=&d(F_{A_1}(x))f+q^{s+i}F_{A_1}(x)d(f)\\ &-q^i(F_{A_1}(d(x))f+q^sF_{A_1}(x)d(f))\\
=&d(F_{A_1}(x))f-q^iF_{A_1}(d(x))f\\
=&(d_{\qhom_{\ma} (X, Y)}(F))_{A_1}(x)f.
\end{align*}
Hence we have $d_{\qhom_{\ma}(X,Y)} \in \Hom_{\ma}^{i+1}(X,Y)$.
\end{proof}

\begin{lemma}\label{homX(A)}
For any right NDG $\ma$-module $X$ and any object $A$ of $\ma$,
\[
\qhom_{\ma}(A\,\hat{}, X)\simeq X(A) ~\mbox{in}~ \cat{C}_{Ndg}(k). 
\]
In particular, $\Hom^{n}_{\ma}(A\,\hat{}, X)\simeq X(A)^n$ for any $n$.
\end{lemma}

\begin{proof}
Define $\phi : \qhom_{\ma}(A\,\hat{}, X) \to X(A)$ by
\[
\phi(F)=F_A(1_A)
\]
for any $F \in \Hom^i(A\,\hat{},X)$, and $\psi : X(A) \to \Hom ^{\bullet q}_{\ma}(A\,\hat{}, X)$ by
\[
\psi(x)_B(a)=xa
\]
for any $x \in X(A)^i$, $a \in A\,\hat{}\,(B)$ and $B \in \ma$. 
\[\begin{aligned}\phi(d_{\qhom_{\ma}(A\,\hat{}, X)}(F))
& = (d_X\circ F -q^i Fd_{A\, \hat{}})_A(1_A) \\
& = d_X(F_A(1_A)) - q^i Fd_A(1_A) \\
& = d_X(F_A(1_A)) \\
& = d_X(\phi(F)), \\
d_X(\psi(x))_B(a)
& = (d_X\circ \psi(x) -q^i \psi(x)d_{A\, \hat{}})_B(a) \\
& = d_X(xa) - q^i xd(a) \\
& = d_X(x)a+q^i xd(a) - q^i xd(a) \\
& = \psi(d(x))_B(a).
\end{aligned}\]
Since $\phi\psi=id$, $\psi\phi=id$,
we have $\Hom ^{\bullet q}_{\ma}(A\,\hat{}, X)\simeq X(A) ~\mbox{in}~ \cat{C}_{Ndg}(k)$.
\end{proof}


\begin{definition}
Let $X$ be an $N$-differential $k$-module.
For $0< r<N$ and $i\in\z$, we define the following $k$-modules:
\[\begin{aligned}
\opn{Z}^i_{(r)}(X) &:=\opn{Ker}(d_X^{i+r-1}\cdots d_X^i), &
\opn{B}^i_{(r)}(X) &:=\opn{Im}(d_X^{i-1}\cdots d_X^{i-r}), \\
\opn{C}^i_{(r)}(X) &:=\opn{Cok}(d_X^{i-1}\cdots d_X^{i-r}), &
\opn{H}^i_{(r)}(X)&:=\opn{Z}^i_{(r)}(X)/\opn{B}^i_{(N-r)}(X) .
\end{aligned}\]
\end{definition}

\begin{lemma}\label{ZBA}
Let $\ma$ be an $N_qDG$ category and $X$, $Y$ and $Z$ right NDG $\ma$-modules.
\begin{enumerate}
\item For any $F \in \opn{Z}^m_{(1)}\Hom _{\ma}^{\bullet q}(X, Y)$
and $G \in \opn{Z}^n_{(1)}\Hom _{\ma}^{\bullet q}(Y, Z)$,
we have that $GF \in \opn{Z}^{m+n}_{(1)}\Hom _{\ma}^{\bullet q}(X, Z)$.
\item For any $F \in \opn{B}^m_{(N-1)}\Hom _{\ma}^{\bullet q}(X, Y)$ and
$G \in \opn{Z}^n_{(1)}\Hom _{\ma}^{\bullet q}(Y, Z)$
we have that $GF \in \opn{B}^{m+n}_{(N-1)}\Hom _{\ma}^{\bullet q}(X, Z)$.
\item For any $F \in \opn{B}^m_{(N-1)}\Hom _{\ma}^{\bullet q}(Y, Z)$ and
$G \in \opn{Z}^n_{(1)}\Hom _{\ma}^{\bullet q}(X, Y))$
we have that $FG \in \opn{B}^{m+n}_{(N-1)}\Hom _{\ma}^{\bullet q}(X, Z)$.
\end{enumerate}
\end{lemma}
\begin{proof}
For any $F \in \opn{Z}^m_{(1)}\Hom _{\ma}^{\bullet q}(X, Y)$
and $G \in \opn{Z}^n_{(1)}\Hom _{\ma}^{\bullet q}(Y, Z)$, we have that
$d(GF)=d(G)F+q^nGd(F)=0$. Therefore $GF \in \opn{Z}^{m+n}_{(1)}\Hom _{\ma}^{\bullet q}(X, Z)$.
\par\noindent
For any $F \in \opn{B}^m_{(N-1)}\Hom _{\ma}^{\bullet q}(X, Y)$ and
$G \in \opn{Z}^n_{(1)}\Hom _{\ma}^{\bullet q}(Y, Z)$, we have that $F=d^{\{N-1\}}(F')$ for some 
$F' \in \Hom _{\ma}^{m-N+1}(X, Y)$ and $d\circ G=q^nG\circ d$.
By Lemma \ref{homtennp} and Lemma \ref{qninv},
\begin{align*}
GF
=&G\circ d^{\{N-1\}}(F')\\
=&G\circ \sum_{l=0}^{N-1}(-1)^lq^{l(m-N+1)+\frac{l(l-1)}{2}}
\begin{bmatrix} N-1\\ l \end{bmatrix}d^{\{N-l-1\}}\circ F' \circ d^{\{l\}}\\
=&\sum_{l=0}^{N-1}(-1)^lq^{l(m-N+1)+\frac{l(l-1)}{2}-n(N-l-1)}
\begin{bmatrix} N-1\\ l \end{bmatrix}d^{\{N-l-1\}}\circ GF' \circ d^{\{l\}}\\
=&q^nd^{\{N-1\}}(GF'). 
\end{align*}
Hence $GF \in \opn{B}^{m+n}_{(N-1)}\Hom _{\ma}^{\bullet q}(X, Z)$.
\par\noindent
Similarly (3) is obtained.
\end{proof}

For an $N_qDG$ category $\ma$, it is easy to see that 
\[
\Hom_{\cat{C}_{Ndg}(\ma)}(X, Y)=\opn{Z}^0_{(1)} \Hom_{\ma}^{\bullet q}(X, Y)
\]
for any $X, Y \in \cat{C}_{Ndg}(\ma) $.

\begin{definition}\label{def:htpcat}
For an $N_qDG$ category $\ma$,
the homotopy category $\cat{K}_{Ndg}(\ma)$ of right $NDG$ $\ma$-modules is defined as follows.
\begin{enumerate}
\item Objects are the right NDG $\ma$-modules.
\item The morphism set between $X$ and $Y$ is given by
\[
\Hom_{\cat{K}_{Ndg}(\ma)}(X, Y)=\opn{H}^0_{(1)} \Hom_{\ma}^{\bullet q}(X, Y)
\]
and the composition is given by the composition of maps (see Lemma \ref{ZBA}).
\end{enumerate}
\end{definition}

\begin{remark}\label{AbNdgA}
For $S \in \Hom_{\ma}^{1-N}(X, Y)$, by Lemmas \ref{qninv}, \ref{homtennp}
\[\begin{aligned}
d^{\{N-1\}}(S)
=&\sum_{l=0}^{N-1}(-1)^lq^{l(-N+1)+\frac{l(l-1)}{2}}
\begin{bmatrix} N-1\\ l \end{bmatrix}d^{\{N-l-1\}}\circ S \circ d^{\{l\}} \\
=&\sum_{l=0}^{N-1}d^{\{N-l-1\}}\circ S \circ d^{\{l\}} .
\end{aligned}\]
Then the above definition of morphisms is equivalent to the definition of homotopy relation in the sense of \cite{Ka}.
\end{remark}


\subsection{Shift functor and suspension functor}\label{shiftsuspension}
In this subsection, we show that
the homotopy category $\cat{K}_{Ndg}(\ma)$ is an algebraic triangulated category and
study the relationship between the shift functor and the suspension functor on $\cat{K}_{Ndg}(\ma)$.

Let $r$ be an integer and $\ma$ an $N_qDG$ category. 
We define the functor 
\[
U_r : \cat{C}_{Ndg}(\ma) \to \cat{Gr}^0(\ma)
\]
by
$U_rX(A)^n=X(A)^{n+r}$ and $U_r(F_A)^n=F_A^{n+r}$ 
for any right NDG $\ma$-module $X$ and $F \in \Hom_{\cat{C}_{Ndg}(\ma)}(X, Y)$.
For any $a \in  \ma(B, A)$ and integer $n$, we denote by $(a^n_{st})_{l}$ an $l \times l$ matrix
whose $s, t$-entry is
\[
a_{st}^n=0~(s>t)~\mbox{and}~\begin{bmatrix} t-1 \\ s-1 \end{bmatrix} q^{n(t-s)}d^{\{t-s\}}(a)~(s\le t).
\]
We define the functor 
\[Q_r : \cat{Gr}^0(\ma) \to \cat{C}_{Ndg}(\ma)
\] 
by
\begin{align*}
(Q_rX)(A)^n=&\coprod_{i=1}^N X(A)^{r+n-N+i}\\
=&\left\{
x=\begin{pmatrix}
x_1 \\ \vdots \\ x_N 
 \end{pmatrix}
~\middle |~x_i \in X(A)^{r+n-N+i}\right\},\\
xa=&{}^t({}^tx(a^n_{st})_{N}), \\
d_{(Q_rX)(A)}(x)=&Jx={}^t({}^tx{}^tJ),~\mbox{where}~
J=
\begin{pmatrix}
0&1&& \\
&\ddots&\ddots& \\
&&\ddots&1 \\
&&&0
\end{pmatrix},
\end{align*}
for any right graded $\ma$-module $X$, $x \in (Q_rX)(A)^n$ and $a \in \ma(B, A)$
(see Lemma \ref{QX} below),

\[
Q_r(F)^n=\begin{pmatrix}
F^{r+n-N+1}&& \\
&\ddots& \\
&&F^{r+n} \end{pmatrix}
: Q_r(X)^n \to Q_r(Y)^n
\]
for any $F \in \Hom_{\cat{Gr}^0(\ma)}(X, Y)$.

\begin{lemma}\label{QX}
Let $\ma$ be an $N_qDG$ category. For any right graded $\ma$-module $X$,
$Q_rX$ is a right NDG $\ma$-module.
\end{lemma}

\begin{proof}
Since
\[
d(xa)
={}^t({}^tx(a^{n}_{st})_N{}^tJ)
\] 
and
\[
d(x)a+q^nxd(a)
={}^t({}^tx{}^tJ(a^{n+1}_{st})_N)
+q^n{\;}^t({}^tx(d(a)^n_{st})_N)
\]
for any $a \in \ma(A,B)$ and $x \in X(B)^n$,
we show that
\[
(a^{n}_{st})_N{}^tJ={}^tJ(a^{n+1}_{st})_N+q^n(d(a)^n_{st})_N.
\]
The $i$, $j$-entry
\[
((a^{n}_{st})_N{}^tJ)_{i\,j}=a^n_{i\, j+1}~\mbox{where we set}~a^m_{i\, N+1}=0
\]
and
\[
({}^tJ(a^{n+1}_{st})_N)_{i\,j}=a^{n+1}_{i-1\, j}~\mbox{where we set}~a^m_{0\, j}=0.
\]
In the case that $i=1$,
\[
a^{n+1}_{i-1\, j}+q^nd(a)^n_{i\,j}=q^nd(a)^n_{i\,j}=a^n_{i\, j+1}.
\]
In the case that $1\le j< i \le N$,
\[
a^{n+1}_{i-1\, j}+q^nd(a)^n_{i\,j}=a^{n+1}_{i-1\, j}=a^n_{i\, j+1}.
\]
In the case that $2\le i\le j \le N$,
\begin{align*}
a^{n+1}_{i-1\, j}+q^nd(a)^n_{i\,j}
=&\begin{bmatrix} j-1 \\ i-2 \end{bmatrix} q^{(n+1)(j-i+1)}d^{\{j-i+1\}}(a)
+q^n\begin{bmatrix} j-1 \\ i-1 \end{bmatrix} q^{n(j-i)}d^{\{j-i\}}(d(a))\\
=&(q^{j-i+1}\begin{bmatrix} j-1 \\ i-2 \end{bmatrix}+\begin{bmatrix} j-1 \\ i-1 \end{bmatrix})
q^{n(j-i+1)}d^{\{j-i+1\}}(a)\\
=&\begin{bmatrix} j \\ i-1 \end{bmatrix} q^{n(j+1-i)}d^{\{j+1-i\}}(a)\\
=&a^n_{i\, j+1}.
\end{align*}
Since
$(xa)b={}^t({}^tx(a^n_{st})_N(b^{n+m}_{st})_N)$ and $x(ab)={}^t({}^tx((ab)^n_{st})_N)$
for any $a \in \ma ^m(B,C)$, $b \in \ma(A,B)$ and $x \in X(C)^n$,
we show that
$(a^n_{st})_N(b^{n+m}_{st})_N=((ab)^n_{st})_N$.
The $i$-$j$ entry
\begin{align*}
((a^n_{st})_N(b^{n+m}_{st})_N)_{i\, j}
=&\sum_{l=0}^{N}a^n_{i\,l}b^{n+m}_{l\, j}\\
=&\sum_{l=i}^{j}\begin{bmatrix} l-1 \\ i-1 \end{bmatrix} q^{n(l-i)}d^{\{l-i\}}(a)
\begin{bmatrix} j-1 \\ l-1 \end{bmatrix} q^{(n+m)(j-l)}d^{\{j-l\}}(b)\\
=&\sum_{l=i}^{j}q^{n(j-i)+m(j-l)}\begin{bmatrix} l-1 \\ i-1 \end{bmatrix}
\begin{bmatrix} j-1 \\ l-1 \end{bmatrix}d^{\{l-i\}}(a)d^{\{j-l\}}(b)\\
=&\sum_{u=0}^{j-i}q^{n(j-i)+mu}\begin{bmatrix} j-u-1 \\ i-1 \end{bmatrix}
\begin{bmatrix} j-1 \\ j-u-1 \end{bmatrix}d^{\{j-u-i\}}(a)d^{\{u\}}(b)\\
=&\sum_{u=0}^{j-i}q^{n(j-i)+mu}\begin{bmatrix} j-1 \\ i-1 \end{bmatrix}
\begin{bmatrix} j-1 \\ u \end{bmatrix}d^{\{j-u-i\}}(a)d^{\{u\}}(b)\\
=&\begin{bmatrix} j-1 \\ i-1 \end{bmatrix} q^{n(j-i)}
\sum_{u=0}^{j-i}q^{mu}\begin{bmatrix} j-i \\ u \end{bmatrix}d^{\{j-i-u\}}(a)d^{\{u\}}(b)\\
=&\begin{bmatrix} j-1 \\ i-1 \end{bmatrix} q^{n(j-i)}d^{\{j-i\}}(ab)~\mbox{(see Lemma \ref{Anp} below.)}\\
=&((ab)^n_{st})_N)_{i\, j}.
\end{align*}
\end{proof}

\begin{lemma}\label{Anp}
For any $f \in \ma ^r(B,C)$ and $g \in \ma(A,B)$,
\[
d^{\{n\}}(fg)=\sum_{l=0}^nq^{lr}\begin{bmatrix} n \\ l \\ \end{bmatrix}d^{\{n-l\}}(f)d^{\{l\}}(g).
\]
\end{lemma}

\begin{proof}
Let $\rho : \ma(B,C)\otimes^{\bullet q} \ma(A,B) \to \ma(A,C)$ in $\cat{C}_{Ndg}(k)$ be the composition of $\ma$.
By Lemma \ref{homtennp},
\begin{align*}
d^{\{n\}}(fg)
=&d^{\{n\}}(\rho(f\otimes g))\\
=&\rho (d^{\{n\}}(f\otimes g))\\
=&\rho (\sum_{l=0}^nq^{lr}\begin{bmatrix} n \\ l \\ \end{bmatrix}d^{\{n-l\}}(f)\otimes d^{\{l\}}(g))\\
=&\sum_{l=0}^nq^{lr}\begin{bmatrix} n \\ l \\ \end{bmatrix}d^{\{n-l\}}(f)d^{\{l\}}(g).
\end{align*} 
\end{proof}

To show Proposition \ref{QUQ}, we prepare the following lemmas.
\begin{lemma}\label{QXAnp}
For $x \in (Q_rX)^n(A)$ and $a \in \ma(B, A)$,
\begin{align*}
d^{\{i\}}(xa)=&\sum_{l=0}^{i}d^{\{i-l\}}(x)a^n_{i-l+1\;i+1}\\
d^{\{N-i\}}(x)a=&\sum_{l=0}^{N-i}d^{\{N-i-l\}}(xa^{n-i}_{i\;i+l})
\end{align*}
\end{lemma}

\begin{proof}
Let $\rho : (Q_rX)(A)\qten  \ma(B, A) \to (Q_rX)(B)$ in $C_{Ndg(k)}$ be
the scalar multiplication of a right $NDG$ $\ma$-module $Q_rX$. 
By Lemma \ref{homtennp},
\begin{align*}
d^{\{i\}}(xa)
=&d^{\{i\}}(\rho(x\otimes a))\\
=&\rho(d^{\{i\}}(x\otimes a))\\
=&\rho (\sum_{l=0}^{i}q^{ln}\begin{bmatrix} i \\ l \\ \end{bmatrix}d^{\{i-l\}}(x)\otimes d^{\{i\}}(a))\\
=&\rho (\sum_{l=0}^{i}d^{\{i-l\}}(x)\otimes a^n_{i-l+1\;i+1})\\
=&\sum_{l=0}^{i}d^{\{i-l\}}(x)a^n_{i-l+1\;i+1}\\
\end{align*}
Set $\phi (x\otimes a)=d(x)\otimes a$ and $\psi (x\otimes a)=q^nx\otimes d(a)$, then 
$\phi +\psi=d_{(Q_rX)(A)\qten  \ma(B, A)}$ and $\psi\phi=q\phi\psi$.
By Lemma \ref{qnnp} (2) and Lemma \ref{homtennp},
\begin{align*}
d^{\{N-i\}}(x)a
=&\rho(\phi^{N-i}(x\otimes a))\\
=&\rho( 
\sum_{l=0}^{N-i} (-1)^lq^{\frac{l(l-1)}{2}}\begin{bmatrix} N-i \\ l \\ \end{bmatrix}(\phi+\psi)^{N-i-l}\psi^l(x\otimes a))\\
=&\rho(
\sum_{l=0}^{N-i} (-1)^lq^{\frac{l(l-1)}{2}}\begin{bmatrix} N-i \\ l \\ \end{bmatrix}d^{\{N-i-l\}}(q^{ln}x\otimes d^{\{l\}}(a)))\\
=&\rho(
\sum_{l=0}^{N-i} \begin{bmatrix} l+i-1 \\ l \\ \end{bmatrix}d^{\{N-i-l\}}(q^{l(n-i)}x\otimes d^{\{l\}}(a)))\\
=&\sum_{l=0}^{N-i}d^{N-i-l}(xa^{n-i}_{i\;i+l}).
\end{align*}

\end{proof}

\begin{lemma}\label{rhoeta}
The following morphisms
\[
\pi_X : Q_{-r}U_rX \to X, \;\pi_{X(A)^m}(x)= (d^{\{N-1\}} \cdots \;\;d \;\;1)x, \;\; x \in Q_{-r}U_rX(A)^m
\]
and
\[
\eta_X : X \to Q_{-r+N-1}U_rX, \;\eta_{X(A)^m}(x)={}^t(1 \;\; d \;\; \cdots \;\; d^{\{N-1\}})x, \;\; x \in X(A)^m
\] 
lie in $\cat{C}_{Ndg}(\ma)$.
\end{lemma}
\begin{proof}
By a direct calculation, both morphisms lie in $\cat{C}_{Ndg}(k)$.
By Lemma \ref{QXAnp},
\begin{align*}
{}^t\eta_X(xa)&=(xa, d(xa), \cdots, d^{\{N-1\}}(xa)) \\
&=(x, d(x), \cdots, d^{\{N-1\}}(x))(a^m_{st})_N\\
&={}^t\eta_X(x)a,
\end{align*}
hence $\eta_X \in \cat{C}_{Ndg}(\ma)$.
For $x={}^t(x_1, \cdots, x_N) \in Q_{-r}U_rX(A)^m$,
by Lemma \ref{QXAnp},
\begin{align*}
\pi_X(x)a&=\Sigma_{i=1}^Nd^{\{N-i\}}(x_i)a \\
&=\Sigma_{i=1}^N\Sigma^{N-i}_{l=0} d^{\{N-i-l\}}(x_ia^{m}_{i\,i+l}) \\
&=\pi_X(xa),
\end{align*}
hence $\pi_X \in \cat{C}_{Ndg}(\ma)$.
\end{proof}

\begin{proposition}\label{QUQ}
For any integer $r$,
$Q_{-r}$ is the left adjoint of $U_r$,
and $U_r$ is the left adjoint of $Q_{-r+N-1}$.
\end{proposition}

\begin{proof}
Let $X$ be a right NDG $\ma$-module and $Y$ a right graded $\ma$-module.
The morphisms
\[
\pi_X : Q_{-r}U_rX \to X, \;\pi_{X(A)^m}(x)= (d^{\{N-1\}} \cdots \;\;d \;\;1)x, \;\; x \in Q_{-r}U_rX(A)^m
\]
in $\cat{C}_{Ndg}(\ma)$ and
\[
\xi_Y : Y \to U_rQ_{-r}Y, \;\eta_{Y(A)^m}(y)={}^t(0 \;\; \cdots \;\; 0 \;\; 1)y, \;\; y \in Y(A)^m
\] 
in $\cat{Gr}^0(\ma)$ imply $Q_{-r}$ is the left adjoint of $U_r$, on the other hand,
the morphisms
\[
\zeta_Y : U_{r}Q_{-r+N-1}Y \to Y, \ \zeta_{Y(A)^m}(y)= (1 \;\; 0 \;\; \cdots \;\; 0)y, \;\; 
y \in U_rQ_{-r+N-1}Y(A)^m
\]
in $\cat{Gr}^0(\ma)$ and
\[
\eta_X : X \to Q_{-r+N-1}U_rX, \;\eta_{X(A)^m}(x)={}^t(1 \;\; d \;\; \cdots \;\; d^{\{N-1\}})x, \;\; x \in X(A)^m
\] 
in $\cat{C}_{Ndg}(\ma)$ imply that $U_r$ is the left adjoint of $Q_{-r+N-1}$.
\end{proof}


Let $\ma$ be an $N_qDG$ category, $\mcal{E}_{\ma}$
the collection of exact sequence $0\to X\to Y\to Z\to 0$ in $\cat{C}_{Ndg}(\ma)$
such that $0\to U_0X\to U_0Y\to U_0Z\to 0$ is a split exact sequence in $\cat{Gr}^0(\ma)$, 
equivalently
$0\to U_rX\to U_rY\to U_rZ\to 0$ is a split exact sequence in $\cat{Gr}^0(\ma)$ for any integer $r$.

\begin{corollary}
$( \cat{C}_{Ndg}(\ma), \mathcal{E}_{\ma})$ is a Frobenius category.
\end{corollary}

\begin{proof}
For any exact sequence $0\to X \to Y \to X \to 0$ in $\mathcal{E}_{\ma}$,
any right graded $\ma$-module $W$ and any integer $r$,
we have the following exact sequence
{\footnotesize
\[
0\to \Hom_{\cat{Gr}^0(\ma)}(U_{r}(Z), W) \to \Hom_{\cat{Gr}^0(\ma)}(U_{r}(Y), W) \to \Hom_{\cat{Gr}^0(\ma)}(U_{r}(X), W)\to 0.
\]}
By Proposition \ref{QUQ}, we have the following exact sequence
{\footnotesize
\[
0\to \Hom_{\cat{C}_{Ndg}(\ma)}(Z, Q_{-r}(W))\to \Hom_{\cat{C}_{Ndg}(\ma)}(Y, Q_{-r}(W))\to
\Hom_{\cat{C}_{Ndg}(\ma)}(X, Q_{-r}(W))\to 0.
\]}
Hence $Q_n(W)$ is injective in $\cat{C}_{Ndg}(\ma)$ for any integer $n$.
Dually one can show that $Q_n(W)$ is projective in $\cat{C}_{Ndg}(\ma)$ for any integer $n$.
For any projective object $X$ and injective object $Y$ of $\cat{C}_{Ndg}(\ma)$,
we have that $\rho_X : Q_{0}U_0X \to X$ is a split epimorphism and 
$\eta_Y : Y \to Q_{N-1}U_0Y$ is a split monomorphism.
Hence  any object of $\cat{C}_{Ndg}(\ma)$ is projective if and only if it is injective.
\end{proof}

\begin{theorem}\label{Kalgtricat}
The stable category of the Frobenius category $( \cat{C}_{Ndg}(\ma), \mathcal{E}_{\ma})$ is
the category $\cat{K}_{Ndg}(\ma)$. In particular, $\cat{K}_{Ndg(\ma)}$ is an algebraic triangulated category.
\end{theorem}
\begin{proof}
It is enough to show that $F \in B^0_{N-1}\Hom_{\ma}^{\bullet q}(X, Y)$ if and only if
$F$ factors through the $\eta_X$ in $\cat{C}_{Ndg}(\ma)$.
For any $F \in B^0_{N-1}\Hom_{\ma}^{\bullet q}(X, Y)$,
$F=d^{\{N-1\}}(H)$ for some $S \in \Hom_{\ma}^{1-N}(X, Y)$.
By Remark \ref{AbNdgA},
\begin{align*}
d^{\{N-1\}}(S) =&\sum_{l=0}^{N-1}d^{\{N-l-1\}}\circ S \circ d^{\{l\}} \\
=&\pi_Y Q_{N-1}(S)\eta_X.
\end{align*}
Hence $F$ factors through the $\eta_X$.

For any $F \in \Hom_{\cat{C}_{Ndg}(\ma)}(X, Y)$,
we assume that $F$ factors through $\eta_X$, namely,
$F=\sum_{l=0}^{N-1}F_l\circ d^{\{l\}}$ for some $(F_0, \cdots , F_{N-1}) \in \Hom_{\cat{C}_{Ndg}(\ma)}(Q_{N-1}U_0(X), Y)$.
By the adjointness $\Hom_{\cat{C}_{Ndg}(\ma)}(Q_{N-1}U_0(X), Y)\simeq \Hom_{\cat{Gr}^0(\ma)}(U_0(X), U_{1-N}(Y))$,
we have that $(F_0, \cdots , F_{N-1})=\zeta_YQ_{N-1}(F')$ for some $F' \in \Hom_{\ma}^{1-N}(X, Y)$.
Hence $F=\zeta_Y Q_{N-1}(F')\eta_X=d^{\{N-1\}}(F')$.
\end{proof}

In the rest of this subsection, we study the shift functor and the suspension functor on $\cat{K}_{Ndg(\ma)}$.
We define the shift functor $\theta_q : \cat{C}_{Ndg}(\ma) \to \cat{C}_{Ndg}(\ma)$ by 
\[ \theta _q(X)(A)^m=X(A)^{m+1}, \]
\[d_{\theta _q(X)(A)}=q^{-1}d_{X(A)} \]
for any right NDG $\ma$-module $(X, d)$ and any object $A$ of $\ma$.
One can check that $d_{\theta _q(X)(B)}(xf)=d_{\theta _q(X)(A)}(x)f+q^mxd(f)$
for any $x \in \theta _q(X)(A)^m$ and $f \in \ma(B, A)$.

We define functors $\Sigma ~(~\mbox{resp.}~ \Sigma^{-1}) : \cat{C}_{Ndg}(\ma) \to \cat{C}_{Ndg}(\ma)$ by
\[
(\Sigma X)^m = \coprod _{i=m+1}^{m+N-1} X^i ~(\mbox{resp.}~(\Sigma^{-1} X)^m = \coprod _{i=m-N+1}^{m-1} X^i),
\]
\[
xa={}^t({}^tx(a^m_{st})_{N-1})\]
and
\[
d_{\Sigma X}^m ~(\mbox{resp.}~d_{\Sigma^{-1} X}^m)= \left( \begin{array}{c|ccccc}
0&1&0&0&\cdots&0\\
&0&\ddots&\ddots&\ddots&\vdots\\
\vdots&\vdots&\ddots&\ddots&\ddots&0\\
&&&\ddots&\ddots&0\\
0&0&\cdots&\cdots&0&1\\
\hline
-d^{\{N-1\}}&-d^{\{N-2\}}&\cdots&\cdots&\cdots&-d\\
\end{array} \right) \] 
for any right NDG $\ma$-module $(X, d)$, 
$x \in \Sigma X(A)^m~(resp.~x \in \Sigma^{-1} X(A)^m)$ and $a \in \ma(B, A)$.

\begin{lemma}\label{Sigmaex}
We have the following short exact sequences which belong to $\mcal{E}_{\ma} $ in $\cat{C}_{Ndg}(\ma)$:
\[\begin{CD}
0 @>>>   \Sigma ^{-1} X @>{\epsilon _X}>>  {Q_0U_0X}  @>{\pi_X}>>  {X} @>>>  0 \\
0 @>>>  X @>{\eta _X}>>  Q_{N-1}U_0X @>{\delta_X}>> \Sigma X  @>>>  0
\end{CD}\]
where
{\footnotesize 
\[ (\epsilon _X )^m = \left( \begin{array}{cccc}
1&0&\cdots&0\\
0&\ddots&\ddots&\vdots\\
\vdots&\ddots&\ddots&0\\
0&\cdots&0&1\\
-d^{\{N-1\}}&\ldots&\ldots&-d \\
\end{array} \right), \quad  
(\pi_X )^m = \left ( \begin{array}{cccc}
d^{\{N-1\}} & \cdots & d& 1 
\end{array} \right)  , \] 
\[ (\eta _X )^m  = \left ( \begin{array}{c}
1\\ d\\ \vdots\\ d^{\{N-1\}}\\
\end{array} \right)  ~\mbox{and}~
(\delta_X )^m = \left( \begin{array}{ccccc}
-d&1&0&\cdots&0\\
0&\ddots&\ddots&\ddots&\vdots\\
\vdots&\ddots&\ddots&\ddots&0\\
0&\ldots&0&-d&1 \\
\end{array} \right) . \]
}
\end{lemma}
\begin{proof}
By a direct calculation, the short exact sequences lie in $\cat{C}_{Ndg}(k)$.
By Lemma \ref{rhoeta}, $\pi_X$ and $\eta_X$ are in $\cat{C}_{Ndg}(\ma)$.
For $x={}^t(0, \cdots,0 , x_i, 0, \cdots, 0)  \in (\Sigma X)^m\;\; (x_i \in X^{m+i-1})$ and $a \in \ma$,
the $t$-entry 
\begin{align*}
\delta_X(xa)_t
=&-d(x_ia^m_{i\,t})+x_ia^m_{i\,t+1}\\
=&-d(x_i)a^m_{i\,t}-q^{m+i-1}x_id(a^m_{i\,t})+x_ia^m_{i\,t+1}\\
=&-d(x_i)a^m_{i\,t}+x_ia^m_{i-1\,t}\\
=&(\delta_X(x)a)_t,
\end{align*}
therefore $\delta_X(xa)=\delta_X(x)a$.

For $x={}^t(0, \cdots,0 , x_i, 0, \cdots, 0)  \in (\Sigma^{-1} X)^m\;\; (x_i \in X^{m-(N-i)})$ and $a \in \ma$,
the $t$-entry 
\begin{align*}
(\epsilon_X(x)a)_t=&x_ia^m_ {i\;t}-d^{N-1}(x_i)a^m_{N\;t},\\
(\epsilon_X(xa))_t=&x_ia^m_{i\;t}~(t \ne N)~\mbox{and} -\sum _{l=1}^{N-1}d^{N-l}(x_ia^m_{i\;l})~(t=N).
\end{align*}
By Lemma \ref{QXAnp}, $\epsilon_X(x)a=\epsilon_X(xa)$.
\end{proof}

The shift functor $\theta_q : \cat{C}_{Ndg}(\ma) \to \cat{C}_{Ndg}(\ma)$ induces 
the shift functor $\theta_q : \cat{K}_{Ndg}(\ma) \to \cat{K}_{Ndg}(\ma)$ which is a triangle functor.
Moreover, $\Sigma$ and $\Sigma ^{-1}$ induces the suspension functor and its quasi-inverse of 
the triangulated category $\cat{K}_{Ndg}(\ma)$.
Then we have the following observation.

\begin{theorem}\label{Sigmatheta}
We have that
\[
\Sigma \simeq \Sigma^{-1} \theta^N_q \simeq \theta^N_q \Sigma^{-1} ~\mbox{in}~ \cat{C}_{Ndg}(\ma),
\]
especially,
\[
\Sigma ^2 \simeq \theta^N_q ~\mbox{in}~ \cat{K}_{Ndg}(\ma).
\]
\end{theorem}

\begin{proof}
In Lemma \ref{Sigmaex}, we have
$\Sigma X=\Sigma^{-1}\theta_q^{N}X$.
\end{proof}

We remark that 
$\Sigma \simeq \Sigma^{-1} \theta^N_q \simeq \theta^N_q \Sigma^{-1}$
and $\Sigma^{-1}$ is not an inverse functor of $\Sigma$ on $\cat{C}_{Ndg(\ma)}$.


\subsection{Bimodules and triangle functors}\label{NDGbimod}
In this subsection,
we give the definition of bimodules over $NDG$ categories and show the adjointness between 
the tensor functor and the hom functor.

\begin{definition}\label{NDG-bimodule}
Let $\ma$ and $\mb$ be $N_qDG$ categories.
An $NDG$ $\mb$-$\ma$-bimodule $M$ is defined by
\begin{enumerate}
\item $(M(A, -), d_{M(A, -)}; \rho_M)$ is a left $NDG$ $\mb$-module for any $A \in \ma$,
\item $(M(-, B), d_{M(-, B)}; \lambda_M)$ is a right $NDG$ $\ma$-module for any $B \in \mb$,
\item we have the following commutative diagram in $\cat{C}_{Ndg}(k)$:
\[\xymatrix{
\mb\qten M\qten \ma \ar[r]^{\ 1\ten\rho_M} \ar[d]^{\lambda_M\ten 1} & \mb\qten M \ar[d]^{\lambda_M} \\
M\qten \ma \ar[r]^{\rho_M} & M .
 }\]
\end{enumerate}
Namely, we have a collection of NDG $k$-modules $(M(A, B), d_{M(A, B)}; \lambda_M, \rho_M)$
for any $A \in \ma$ and $B \in \mb$ satisfying that 
$(M(A, -), d_{M(A, -)}; \rho_M)$ is a left $NDG$ $\mb$-module for $A \in \ma$,
$(M(-, B), d_{M(-, B)}; \lambda_M)$ is a right $NDG$ $\ma$-module for $B \in \mb$ and
$f(mg)=(fm)g$ for $f \in \mb(B_1, B_2)$, $m \in M(A_1, B_1)$ and $g \in \ma(A_2, A_1)$.
One can check that $d(f(mg))=d((fm)g)$.
\end{definition}

\begin{example}
Let $\ma$ be an $N_qDG$ category. Then $\ma$ is an $NDG$ $\ma$-$\ma$-bimodule.
We take $k$ as an $N_qDG$ category with one object $*$ and the morphism set $k(*, *)=k^0(*, *)=k$.
Any right (resp. left) $NDG$ $\ma$-module $X$ is an $NDG$ $k$-$\ma$-bimodule 
(resp. $NDG$ $\ma$-$k$-bimodule).
\end{example}

Let $\ma$ and $\mb$ be small $N_qDG$ categories,
$X$ a right $NDG$ $\mb$-module,
$Y$ a right $NDG$ right $\ma$-module and
$M$ an $NDG$ $\mb$-$\ma$-bimodule.
We define a right $NDG$ $\mb$-module $\Hom_{\ma}^{\bullet q}(M, Y)$ by
\[
\Hom_{\ma}^{\bullet q}(M, Y)(B)=\Hom_{\ma}^{\bullet q}(M(-, B), Y)
\]
and
\[
\rho : \Hom_{\ma}^{\bullet q}(M, Y)(B_1)\qten \mb(B_2, B_1) \to \Hom_{\ma}^{\bullet q}(M, Y)(B_2)
\] 
\[\rho(F\otimes b)_A(m)=F_A(bm)
\]
for any $B$, $B_1$ and $B_2 \in \mb$, $m \in M(A, B_2)$ and $A \in \ma$.
One can check that $\rho$ is in $C_{Ndg}(k)$.
We define a right $NDG$ $\ma$-module $X\qten_{\mb}M$ by
\[
(X\qten_{\mb}M)(A)={\rm Cok}\,\nu_A
\] 
in $C_{Ndg}(k)$ for any $A$ in $\ma$
where
\[
\nu_A : \coprod_{B_1, B_2 \in \mb}X(B_2)\qten \mb (B_1, B_2)\qten M(A, B_1)
\to \coprod_{B_3 \in \mb}X(B_3)\qten M(A, B_3),
\]
\[
\nu_A(x\otimes b\otimes m)=xb\otimes m-x\otimes bm
\]
and
\[
\rho : X(B)\qten_{\mb}M(A_1, B)\qten \ma(A_2, A_1) \to X(B)\qten_{\mb}M(A_2, B)\] 
\[
\rho(x \otimes m \otimes a)=x \otimes ma
\]
for any $A_1$, $A_2 \in \ma$ and $B \in \mb$.
One can check that $\mu_A$ and $\rho$ are in $C_{Ndg}(k)$.
Similarly, a left $NDG$ $\ma$-module $M\qten_{\mb}Y$
is defined for any left $NDG$ $\mb$-module $Y$.

It is easy to see the following Remark and Lemmas.

\begin{remark}\label{exhomtenNdgA}
$\qhom_{\ma}(M,-): \cat{C}_{Ndg}(\ma) \to \cat{C}_{Ndg}(\mb)$ (resp., $\Hom_{\cat{Gr}(\ma)}(M,-): \cat{Gr}^0(\ma) \to \cat{Gr}^0(\mb)$) is left exact
between abelian categories in the sense of Remark \ref{AbNdgA}.
And $-\qten_{\mb}M : \cat{C}_{Ndg}(\mb) \to \cat{C}_{Ndg}(\ma)$ (resp., $(-\qten_{\mb}M)\mid_{\cat{Gr}^0(\mb)} : \cat{Gr}^0(\mb) \to \cat{Gr}^0(\ma)$) is right exact between abelian categories, where $(-\qten_{\mb}M)\mid_{\cat{Gr}^0(\mb)}$ is the restriction of $-\qten_{\mb}M$ to $\cat{Gr}^0(\mb)$.
\end{remark}


\begin{lemma}\label{qtenthate}
Let $\ma$ and $\mb$ be small $N_qDG$ categories,
$X$ a right $NDG$ $\mb$-module, and $M$ an $NDG$ $\mb$-$\ma$-bimodule.
For any $n\in \z$, we have the canonical isomorphism in $\cat{C}_{Ndg}(\ma)$:
\[
(\theta_qX)\qten_{\mb}M_{\ma} \simeq \theta_q(X\qten_{\mb}M_{\ma})
\] 

\end{lemma}

\begin{lemma}\label{tenM(B)}
Let $A$ (resp., $B$) be an object of $\ma$ (resp., $\mb$), $M$ an $NDG$ $\mb$-$\ma$-bimodule.
Then the following hold.
\begin{enumerate}
\item
$B\,\hat{}\qten_{\mb}M \simeq M(-, B)$ as right $NDG$ $\ma$-modules. 
\item
$M\qten_{\ma}\hat{}A \simeq M(A, -)$  as left $NDG$ $\mb$-modules. 
\end{enumerate}
\end{lemma}

\begin{proof}
(1)
It is easy to see that
{\small
\[
\coprod_{B_1, B_2 \in \mb}\mb(B_2,B)\qten\mb(B_1,B_2)\qten M(-, B_1) 
\xarr{\nu}
\coprod_{B_1 \in \mb}\mb(B_1,B)\qten M(-, B_1) 
\xarr{\pi}
M(-, B)
\]
}
is a composition of morphisms in $\cat{C}_{Ndg}(\ma)$, where $\nu=1\ten \lambda_M -\mu_{\mb}\ten 1$, $\pi=\Sigma \lambda^{B_1}_{B}$ 
with $\lambda^{B_1}_{B}:\mb(B_1,B)\qten M(-, B_1)\to M(-, B)$.
Then it suffices to show that the following sequence is exact in $\cat{C}_{Ndg}(k)$:
{\small
\[
\coprod_{B_1, B_2 \in \mb}\mb(B_2,B)\qten\mb(B_1,B_2)\qten M(A, B_1) 
\xarr{\nu}
\coprod_{B_1 \in \mb}\mb(B_1,B)\qten M(A, B_1) 
\xarr{\pi}
M(A, B)
\to
0
\]
}
for any $A \in \ma$.
It is trivial that $\pi=\Sigma \lambda^{B_1}_{B}$ is an epimorphism.
Since $\lambda^{B_1}_{B}\circ(\mu_{\mb}\ten 1)=\lambda^{B_2}_{B}\circ(1\ten \lambda^{B_1}_{B})$, we have $\pi\circ\nu=0$.
Then there is an epimorphism $\alpha$ such that
$\pi=\alpha\circ\opn{cok}(\nu)$.
For any finite set $\{B_1, \cdots, B_{n}\}$ of objects of $\mb$ with $B_1=B$,
consider an $n\times n$ matrix ring $\Lambda={}^{t}(\mb(B_i,B_j))_{1\leq i, j \leq n}$, a left $\Lambda$-module $\tilde{M}=\coprod_{i=1}^{n}M(A, B_i)$,
then $M(A, B)=e_1\Lambda\ten_{\Lambda}\tilde{M}$, where $e_1$ is the idempotent corresponding to $B_1$.
Then we have the following exact sequence
{\small
\[
\coprod_{1 \leq i, j \leq n}\mb(B_j,B)\qten\mb(B_i,B_j)\qten M(A, B_i) 
\xarr{\tilde{\nu}}
\coprod_{1 \leq i \leq n}\mb(B_i,B)\qten M(A, B_i) 
\xarr{\tilde{\pi}}
M(A,B)
\to
0
\]
}
where $\tilde{\nu}$ is a restriction of $\nu$.
Since $\coprod_{B_1 \in \mb}\mb(B_1,B)\qten M(A, B_1)$ is the inductive limit  $\underset{\rightarrow}{\lim} \coprod_{B_i \in \cat{B}}\mb(B_i,B)\qten M(A, B_i)$ where $\cat{B}$ is a finite set of objects which contains $B$,
there is an epimorphism $\beta$ such that $\opn{cok}(\nu)=\beta\circ\pi$.
Hence we have the above exact sequence in $\cat{C}_{Ndg}(k)$.
\par\noindent
(2)
Similarly.
\end{proof}

\begin{theorem}\label{adjNDG}
Let $\ma$ and $\mb$ be small $N_qDG$ categories. The following hold for a right $NDG$ $\mb$-module $X$,
a right $NDG$ $\ma$-module $Y$ and an $NDG$ $\mb$-$\ma$-bimodule $M$.
\begin{enumerate}
\item $\qhom _{\ma}(X\qten_{\mb} M, Y) \simeq 
\qhom _{\mb}(X, \qhom_{\ma}(M, Y))$ in $\cat{C}_{Ndg}(k)$.

\item $\Hom _{\cat{C}_{Ndg}(\ma)}(X\qten_{\mb} M, Y) \simeq 
\Hom _{C_{Ndg(\mb)}}(X, \qhom_{\ma} (M, Y))$.

\item $\Hom _{K_{Ndg(\ma)}}(X\qten_{\mb} M, Y) \simeq 
\Hom _{K_{Ndg(\mb)}}(X, \qhom_{\ma} (M, Y))$.
\end{enumerate}
\end{theorem}

\begin{proof}
Define
$\alpha : \Hom _{\ma}^{\bullet q}(X\otimes_{\mb}^{\bullet q} M, Y) \to
\Hom _{\mb}^{\bullet q}(X, \Hom_{\ma}^{\bullet q} (M, Y))$ as
\[
(\alpha (F)_B(x))_A(m)=F_A(x \otimes m)
\]
for any $F \in \Hom _{\ma}^{\bullet q}(X\otimes_{\mb}^{\bullet q} M, Y)$,
$x \in X(B)$ and $m \in M(A, B)$.
For any $F \in \Hom _{\ma}^{s}(X\otimes_{\mb}^{\bullet q} M, Y)$,
any $g \in \mb(B',B)$,
\[\begin{array}{lll}
(\alpha(F)_{B'}(xg))_A(m) & = F_A(xg\ten m) 
& = F_A(x\ten gm) \\
& = \alpha(F)_B(x)(gm) 
& = ((\alpha(F)_B(x))_Ag)(m) .
\end{array}\]
Therefore $\alpha (F) \in \Hom_{\mb}^s(X, \qhom_{\ma}(M, Y))$.
\par\noindent
For any $F \in \Hom _{\ma}^{s}(X\qten_{\mb}M, Y)$,
$x\in X(B)^t$ and $m \in M(A, B)$,
\[\begin{array}{lll}
(\alpha (d(F))_B(x))_A(m) 
& = &(\alpha(d\circ F - q^s F\circ d)_B(x))_A(m) \\
& = &(d\circ F-q^s F\circ d)_A(x\ten m) \\
& = & d\circ F_A(x\ten m)-q^s F_A(d(x)\ten m+q^tx\ten d(m)) \\
& = & d\circ F_A(x\ten m)-q^s F_A(d(x)\ten m)-q^{s+t}F_A(x\ten d(m)), \\
\\
(d(\alpha(F))_B(x))_A(m) 
& = & d(\alpha(F)_B(x))_A(m)-q^s(\alpha(F)_B(d(x)))_A(m) \\
& = & d\circ F_A(x\ten m) -q^{s+t}F_A(x\ten d(m))-q^sF_A(d(x)\ten m).
\end{array}\]
Then $\alpha \circ d=d\circ \alpha$, so that $\alpha$ is in $C_{Ndg}(k)$, and $\alpha$ is well defined.
Therefore it suffices to show that $\alpha: \Hom_{\cat{Gr}(\mb)}(X\otimes_{\mb}^{\bullet q} M, Y) \to \Hom_{\cat{Gr}(\ma)}(X, \Hom_{\ma}^{\bullet q} (M, Y))$
is an isomorphism in $\cat{Gr}^0(k)$.
Let $X=\theta_q^iB\,\hat{}$.
By Lemmas \ref{qtenthate}, \ref{tenM(B)}, $\alpha$ is an isomorphism:
\[\begin{aligned}
\Hom_{\cat{Gr}(\ma)}(\theta_q^iB\,\hat{}\qten_{\mb}M,Y)& \xarr{\sim}
\Hom_{\cat{Gr}(\ma)}(\theta_q^iM(-, B),Y) \\
& \xarr{\sim}
\Hom_{\cat{Gr}(\mb)}(\theta_q^iB\,\hat{},\qhom_{\ma}(M,Y)).
\end{aligned}\]
According to Lemma \ref{homX(A)}, we have a projective presentation of $X$ in $\cat{Gr}^0(\mb)$:
\[
\coprod_{i_1 \in \z}\coprod_{B_1}\theta_q^{i_1}{B_1}\,\hat{}\ {}^{(I_{i_1})}
\to
\coprod_{i_0 \in \z}\coprod_{B_0}\theta_q^{i_0}{B_0}\,\hat{}\ {}^{(I_{i_0})}
\to X
\to
0 .
\]
By applying $-\qten_{\mb}M$ to the above, we have
an exact sequence in $\cat{Gr}^0(\ma)$:
\[
\coprod_{i_1 \in \z}\coprod_{B_1}\theta_q^{i_1}{B_1}\,\hat{}\ {}^{(I_{i_1})} \qten_{\mb}M
\to
\coprod_{i_0 \in \z}\coprod_{B_0}\theta_q^{i_0}{B_0}\,\hat{}\ {}^{(I_{i_0})} \qten_{\mb}M
\to X \qten_{\mb}M
\to
0 .
\]
Since $\Hom_{\cat{Gr}(\mb)}(-, \Hom_{\ma}^{\bullet q} (M, Y))$ and $\Hom_{\cat{Gr}(\ma)}(-, Y)$ preserve the left exactness of the above exact sequences, 
by Remark \ref{exhomtenNdgA} and Lemmas \ref{qtenthate}, \ref{tenM(B)},
\[
\alpha:
\Hom _{\cat{Gr}(\ma)}(X\qten_{\mb} M, Y)
\xarr{\sim}
\Hom _{\cat{Gr}(\mb)}(X, \qhom_{\ma}(M, Y))
\]
is an isomorphism.
Thus one can show (2) and (3).
\end{proof}

\begin{corollary}\label{exactfunc}
Let $\ma$ and $\mb$ be small $N_qDG$ categories. The following hold for an $NDG$ $\mb$-$\ma$-bimodule $M$.
\begin{enumerate}
\item $\qhom _{\ma}(M, -):\cat{K}_{Ndg}(\ma) \to \cat{K}_{Ndg}(\mb)$ is a triangle functor.
\item $-\qten_{\mb} M:\cat{K}_{Ndg}(\mb) \to \cat{K}_{Ndg}(\ma)$ is a triangle functor.
\end{enumerate}
\end{corollary}

\begin{proof}
By Proposition \ref{extr01}, it suffices to show the above functors send projective objects to projective objects.
Consider 
\[
\Hom _{K_{Ndg(\ma)}}(X\qten_{\mb} M, Y) \simeq 
\Hom _{K_{Ndg(\mb)}}(X, \qhom_{\ma} (M, Y))
\]
If $Y =0$ in $\cat{K}_{Ndg}(\ma)$, then the left side is equal to $0$ for any $X \in \cat{K}_{Ndg}(\mb)$.
Therefore $\qhom_{\ma}(M,Y)=0$  in $\cat{K}_{Ndg}(\mb)$.
Similarly $X\qten_{\mb} M=0$ in $\cat{K}_{Ndg}(\ma)$ if $X= 0$ in $\cat{K}_{Ndg}(\mb)$.
\end{proof}


\section{Derived categories of $NDG$ modules}\label{DGNDG}
In this section, we study the derived category of $NDG$ modules and
establish triangles equivalence of Morita type for $NDG$ categories.
Furthermore, we show that any derived category of $NDG$ modules have a set of symmetric generators,
and then describe derived equivalences induced by $NDG$ bimodules.

\begin{definition}\label{NDghomology}
For an $N_qDG$ category $\ma$,
let $\cat{Z}^0_{(1)}\ma$ be the category whose objects are the same as $\ma$ and the morphism set between $A_1$ and $A_2$ is given by $\cat{Z}^0_{(1)}\ma(A_,A_2)=\opn{Z}^0_{(1)}\ma(A_1, A_2)$, $\Mod(\cat{Z}^0_{(1)}\ma)$ the category of contravariant $k$-linear functors from $\cat{Z}^0_{(1)}\ma$ to the category of $k$-modules.
For $0< r<N$ and $i\in\z$, we define the following covariant functors from $\cat{C}_{Ndg}(\ma)$ to $\Mod(\cat{Z}^0_{(1)}\ma)$:
\[\begin{aligned}
\opn{Z}^i_{(r)}(X) &:=\opn{Ker}(d_X^{i+r-1}\cdots d_X^i), &
\opn{B}^i_{(r)}(X) &:=\opn{Im}(d_X^{i-1}\cdots d_X^{i-r}), \\
\opn{C}^i_{(r)}(X) &:=\opn{Cok}(d_X^{i-1}\cdots d_X^{i-r}), &
\opn{H}^i_{(r)}(X)&:=\opn{Z}^i_{(r)}(X)/\opn{B}^i_{(N-r)}(X) .
\end{aligned}\]
\end{definition}

We call $X$ in $\cat{C}_{Ndg}(\ma)$  an $N$-acyclic if $\h^i_{(r)}(X)=0$ in  $\Mod(\cat{Z}^0_{(1)}\ma)$ for any $i \in \z$, $1\leq r \leq N-1$.
According to \cite[Proposition 1.5]{Ka}, this is enough for $r=1$.
We call a morphisms $F : X\to Y$ in $\cat{C}_{Ndg}(\ma)$ a quasi-isomorphism if
$F$ induces an isomorphism $\h^i_{(r)}(F):\h^i_{(r)}(X)\iso \h^i_{(r)}(Y)$ in $\Mod(\cat{Z}^0_{(1)}\ma)$ for any $i \in \z$, $1\leq r \leq N-1$.
By the following Lemma, this is enough for $r=1, N-1$.
We denote by $\cat{K}_{Ndg}^{\phi}(\ma)$ the full subcategory of $\cat{K}_{Ndg}(\ma)$ 
consisting of $N$-acyclic right $NDG$ $\ma$-modules. 
According to the hexagon of homologies in \cite{D} (see Lemma \ref{hghexagon}), $\cat{K}_{Ndg}^{\phi}(\ma)$ is a full triangulated subcategory of $\cat{K}_{Ndg}(\ma)$ which is closed under direct summands.
The derived category of right $NDG$ $\ma$-modules
is defined by the quotient category
\[\cat{D}_{Ndg}(\ma)=\cat{K}_{Ndg}(\ma)/K^{\phi}_{Ndg}(\ma).
\]

\begin{lemma}\label{hghexagon}
The following hold in $\cat{K}_{Ndg}(\ma)$.
\begin{enumerate}
\item
For a triangle $X \xarr{F} Y \xarr{G} Z \to \Sigma X$ we have the following exact sequence
\[
\begin{array}{llll}
\cdots& \to&\h^{i}_{(r)}(X) \to \h^{i}_{(r)}(Y) \to\h^{i}_{(r)}(Z)&\\
& \to & \h^{i+r}_{(N-r)}(X)\to \h^{i+r}_{(N-r)}(Y) \to \h^{i+r}_{(N-r)}(Z)&\\
&\to &\h^{i+N}_{(r)}(X)\to\cdots .&
\end{array}
\]
\item $\opn{H}^{i}_{(r)}\Sigma X \simeq \opn{H}^{i+r}_{(N-r)}X$ for any $i$, $1 \leq r < N$.
\item A morphism $F: X \to Y$ is a quasi-isomorphism
if and only if $\opn{H}^i_{(1)}(F)$ and $\opn{H}^i_{(N-1)}(F)$ are isomorphisms for any $i$.
\end{enumerate}
\end{lemma}

\begin{proof}
(1) 
By the hexagon of homologies in \cite{D}.
\par\noindent
(2)
Since $X \to Q_{N-1}U_0(X) \to \Sigma X \xarr{1} \Sigma X$ is a triangle, and $\opn{H}^i_{(r)}Q_{N-1}U_0(X) = 0$ for any $i$, $1 \leq r < N$,
$\opn{H}^{i}_{(r)}\Sigma X \simeq \opn{H}^{i+r}_{(N-r)}X$.
\par\noindent
(3)
In the triangle of (1), if $\opn{H}^i_{(r)}(F)$ is an isomorphism
for any $i$, $r=1, N-1$, then $\opn{H}^i_{(r)}Z=0$ for any $i$, $r=1, N-1$.
According to \cite[Proposition 1.5]{Ka}, 
$\opn{H}^i_{(r)}Z=0$ for any $i$, $1 \leq r < N-1$.
\end{proof}

\begin{example}\label{NDGk1}
Since $\cat{C}_{Ndg}(k)$ is the category $\cat{C}_{N}(\Mod k)$ of complexes of $k$-modules,
$\cat{K}_{Ndg}(k)$ is the homotopy category $\cat{K}_{N}(\Mod k)$ of complexes of $k$-modules.
Then we have $\cat{D}_{Ndg}(k)=\cat{D}_{N}(\Mod k)$.
In this case, we have the following pull-back square $(D_{(r)}^{n})$ (see \cite{IKM3}):
\[\xymatrix{
0 \ar[r] & \opn{Z}_{(1)}^{n}(X) \ar@{=}[d] \ar[r] & \ar @{} [dr] |{(D_{(r)}^{n})} \opn{Z}_{(r)}^{n}(X) \ar[r]^{d} \ar@{^{(}->}[d] 
& \opn{Z}_{(r-1)}^{n+1}(X) \ar@{^{(}->}[d] \\
0 \ar[r] & \opn{Z}_{(1)}^{n}(X) \ar[r] & \opn{Z}_{(r+1)}^{n}(X) \ar[r]^{d} & \opn{Z}_{(r)}^{n+1}(X)
}\]
\end{example}

\begin{lemma}\label{NDGk2}
In an abelian category for a commutative diagram
\[\xymatrix{
& V_1 \ar@{^(->}[d]^{\left(\begin{smallmatrix} 1 \\ 0 \end{smallmatrix}\right)}\ar@{=}[r] & V_1 \ar@{^(->}[d] ^{\left(\begin{smallmatrix} 1 \\ 0 \\ 0 \end{smallmatrix}\right)} \\
0 \ar[r] & V_1\oplus V_2 \ar@{->>}[d] ^{\left(\begin{smallmatrix} 0 & 1 \end{smallmatrix}\right)}\ar[r]^{\left(\begin{smallmatrix} 1 & 0 \\ 0 & 1 \\ 0 & 0 \end{smallmatrix}\right)} & V_1\oplus V_2\oplus V_3 \ar @{->>}[d]^{\left(\begin{smallmatrix} 0 & 1 & f \\ 0 & 0 & 1 \end{smallmatrix}\right)}\ar[r]^{\qquad\left(\begin{smallmatrix} 0 & 0 & 1 \end{smallmatrix}\right)} & V_3 \ar@{=}[d]\ar[r] & 0 \\
0 \ar[r] & V_2 \ar[r]^{\left(\begin{smallmatrix} 1 \\ 0 \end{smallmatrix}\right)} & V_2\oplus V_3 \ar[r]^{\quad\left(\begin{smallmatrix} 0 & 1 \end{smallmatrix}\right)} & V_3 \ar[r] & 0 \\
}\]
where all rows and columns are exact.
Via an isomorphism $\left(\begin{smallmatrix} 1 & 0 & 0 \\ 0 & 1 & f \\ 0 & 0 & 1 \end{smallmatrix}\right): V_1\oplus V_2\oplus V_3 \to V_1\oplus V_2\oplus V_3$,
we can replace $f$ by $0$ in the above.
\end{lemma}

\begin{proposition}\label{NDGk}
If $k$ is a field, then $\cat{D}_{Ndg}(k)=\cat{K}_{Ndg}(k)$.
\end{proposition}

\begin{proof}
It suffices to show that every NDG $k$-module $X$ of which all homologies are null is a zero object of $\cat{K}_{Ndg}(k)$.
Since $\opn{Z}_{(r)}^{i}(X)=\opn{B}_{(N-r)}^{i}(X)$ for any $i$, $1 \leq r \leq N-1$,
we have an exact square
\[\xymatrix{
\ar @{} [dr] |{(D_{(r)}^{n})} \opn{Z}_{(r)}^{n}(X) \ar@{^{(}->}[r] \ar@{->>}[d]^{d} 
& \opn{Z}_{(r+1)}^{n}(X) \ar@{->>}[d]^{d} \\
\opn{Z}_{(r-1)}^{n+1}(X) \ar@{^{(}->}[r] & \opn{Z}_{(r)}^{n+1}(X)
}\]
where all rows (resp., columns) are monic (resp., epic)
for any $n$, $1 \leq r < N$.
Here $\opn{Z}_{(0)}^{n+1}(X) = 0$ and $\opn{Z}_{(N)}^{n+1}(X) = X^n$.
By applying Lemma \ref{NDGk2} to the above diagram for $1\leq r < N$, it is easy to see that
$X$ is isomorphic to
$(Y, d_Y)$ in $\cat{C}_{Ndg}(k)$, where
$Y^n=\coprod_{i=0}^{N-1}\opn{Z}_{(1)}^{n+i}(X)$,
$d_Y$ is an $N\times N$ matrix $\begin{pmatrix}
0&1&& \\
&\ddots&\ddots& \\
&&\ddots&1 \\
&&&0
\end{pmatrix}$.
Therefore $X$ is a zero object of $\cat{K}_{Ndg}(k)$.
\end{proof}

\begin{proposition}\label{HhomA}
For any right NDG $\ma$-module $X$, $Y$ and any object $A$ of $\ma$,
\begin{enumerate}
\item $\Hom_{K_{Ndg(\ma)}}(\theta^{-n}_qX, Y)\simeq \h ^n_{(1)}\Hom_{\ma}^{\bullet q}(X, Y),$
\item $\Hom_{K_{Ndg(\ma)}}(\theta^{-n}_qX, \Sigma Y)\simeq \h ^n_{(N-1)}\Hom_{\ma}^{\bullet q}(X, Y).$
\end{enumerate}
In particular, by Lemma \ref{homX(A)},
\[
\Hom_{K_{Ndg(\ma)}}(\theta^{-n}_qA\,\hat{}, X)\simeq  \h ^n_{(1)}X(A).
\]
\end{proposition}
\begin{proof}
We have that $Z^0_{(1)}\qhom _{\ma}(\theta^{-n}_qX, Y)=Z^n_{(1)}\qhom _{\ma}(X, Y)$
because $F$ is in $Z^0_{(1)}\qhom _{\ma}(\theta^{-n}_qX, Y)$ if and only if
$d_Y\circ F=F\circ d_{\theta^{-n}_qX}=q^{n}F\circ d_X$ if and only if
$F$ is in $Z^n_{(1)}\qhom _{\ma}(X, Y)$.
\par\noindent
For any $F \in \opn{B}^n_{(N-1)}\qhom _{\ma}(X, Y)$
we have $F=d^{\{N-1\}}_{\qhom _{\ma}(X, Y)}(F')$ 
for some $F' \in \Hom^{n+1-N}_{\ma}(X, Y)=\Hom^{1-N}_{\ma}(\theta^{-n}_qX, Y)$.
By Lemmas \ref{homtennp}, \ref{qninv},
\begin{align*}
d^{\{N-1\}}_{\qhom _{\ma}(X, Y)}(F')
=&\sum_{l=0}^{N-1}(-1)^lq^{l(n+1-N)+\frac{l(l-1)}{2}}
\begin{bmatrix} N-1\\ l \end{bmatrix}d^{\{N-l-1\}}_Y\circ F' \circ d^{\{l\}}_X \\
=&\sum_{l=0}^{N-1}d^{\{N-l-1}\}_{Y}\circ F'\circ q^{nl}d^{\{l\}}_X \\
=&\sum_{l=0}^{N-1}d^{\{N-l-1\}}_{Y}\circ F' \circ d^{\{l\}}_{\theta^{-n}_qX} \\
=&d^{\{N-1\}}_{\qhom _{\ma}(\theta^{-n}_qX, Y)}(F').
\end{align*}
Therefore $\opn{B}^0_{(N-1)}\qhom _{\ma}(\theta^{-n}_qX, Y)=\opn{B}^n_{(N-1)}\qhom _{\ma}(X, Y)$.
Hence 
\[
\Hom_{K_{Ndg(\ma)}}(\theta^{-n}_qX, Y)=\h ^0_{(1)}\qhom _{\ma}(\theta^{-n}_qX, Y)
=\h ^n_{(1)}\qhom _{\ma}(X, Y).
\]
\par\noindent
One can show that $\h ^n_{(1)}\qhom _{\ma}(X, \Sigma Y)\simeq \h ^n_{(N-1)}\qhom _{\ma}(X, Y)$.
Thus (2) is proved.
\end{proof}

Krause introduced the notion of symmetric generators (\cite{Kr0}).
Let $E$ be an injective cogenerator of $k$-modules and $\ma$ an $N_qDG$ category.
For a left $NDG$ $\ma$-module $X$, we denote by $\md X$ a right $NDG$ $\ma$-module
defined by $\md X(A)=\qhom(X(A), E)$ for any $A \in \ma$.

\begin{definition}
A set $\cat{U}$ of objects in a triangulated category $\mcal{D}$ is called a set of symmetric generators
for $\mcal{D}$ if the following hold:
\begin{enumerate}
\item $\Hom_{\mcal{D}}(\cat{U}, X)=0$ implies $ X=0$.
\item There is a set $\cat{V}$ of objects in $\mcal{C}$ such that
for any $X \to Y$ in $\mcal{D}$ the induced morphism $\Hom_{\mcal{D}}(U,X) \to \Hom_{\mcal{D}}(U,Y)$
is surjective for any $U \in \cat{U}$ if and only if 
$\Hom_{\mcal{D}}(Y,V) \to \Hom_{\mcal{D}}(X,V)$ is injective for any $V \in \cat{V}$.
\end{enumerate}
\end{definition}

\begin{theorem}\label{gencogen}
Let $\ma$ be a small $N_qDG$ category.
Let $\cat{U}$ be a set $\{ \Sigma ^i \theta^j_q A\,\hat{}\; |\; A\in \ma, 1-N \leq j \leq 0, i \in \mathbb{Z} \}$.
Then the following hold.
\begin{enumerate}
\item
$\cat{U}$
is a set of compact generators for $\cat{D}_{Ndg}(\ma)$.
\item
$\cat{U}$
is a set of symmetric generators for $\cat{D}_{Ndg}(\ma)$. 
\end{enumerate}
\end{theorem}

\begin{proof}
(1)
For any $n \in \z$, by Theorem \ref{Sigmatheta},
$\theta^n_q A\,\hat{}=\Sigma ^i \theta^j_q A\,\hat{}$ for some $i \in \z$ and $1-N \le j \le 0$.
By Proposition \ref{HhomA}, 
$\Hom_{\cat{K}_{Ndg}(\ma)}(\theta^n_q A\,\hat{}, X)=0$ for any $N$-acyclic $X$.
Thus, for any $Y\in \cat{D}_{Ndg}(\ma)$,
\begin{align*}
\Hom_{\cat{D}_{Ndg}(\ma)}(\theta^n_q A\,\hat{}, Y)
\simeq &\Hom_{\cat{K}_{Ndg}(\ma)}(\theta^n_q A\,\hat{}, Y)\\
\simeq &\h^{-n}_{(1)}(Y)(A)
\end{align*}
and, for any $Y_i \in \cat{D}_{Ndg}(\ma)\;(i \in I)$,
\begin{align*}
\Hom_{\cat{D}_{Ndg}(\ma)}(\Sigma ^m\theta^n_q A\,\hat{}, \coprod _{i\in I}Y_i)
\simeq &\h^{-n}_{(1)}(\Sigma ^{-m}\coprod _{i\in I}Y_i)(A)\\
\simeq &\coprod _{i\in I}\h^{-n}_{(1)}(\Sigma ^{-m}Y_i)(A)\\
\simeq &\coprod _{i\in I}\Hom_{\cat{D}_{Ndg}(\ma)}(\Sigma ^{m}\theta^n_q A\,\hat{}, Y_i).
\end{align*}
Hence (1) is proved.
\par\noindent
(2)
Consider $\cat{V}=\{ \Sigma ^i \theta^j_q \mathbb{D}~\hat{}A\; |\; A\in \ma, 0 \leq j \leq N-1, i \in \z \}$.
For any $n \in \z$, by Theorem \ref{Sigmatheta},
$\theta^n_q \mathbb{D}~\hat{}A=\Sigma ^i \theta^j_q \mathbb{D}~\hat{}A$ for some $i \in \z$ and $0 \le j \le N-1$.
By Proposition \ref{HhomA} and Lemma \ref{tenM(B)},
\begin{align*}
\Hom_{\cat{K}_{Ndg}(\ma)}(X, \theta^n_q \mathbb{D}~\hat{}A)
\simeq &\Hom_{\cat{K}_{Ndg}(\ma)}(\theta^{-n}_q X, \mathbb{D}~\hat{}A)\\
\simeq &\h^n_{(1)}\Hom_{\ma}^{\bullet q}(X, \Hom^{\bullet q}(\;\hat{}A, E))\\
\simeq &\h^n_{(1)}\Hom^{\bullet q}(X\qten_{\ma}\hat{}A, E)\\
\simeq &\h^n_{(1)}\Hom^{\bullet q}(X(A), E)\\
\simeq &\Hom(\h^n_{(1)}X(A), E).
\end{align*}
Therefore $\Hom_{\cat{K}_{Ndg}(\ma)}(X, \theta^n_q \mathbb{D}~\hat{}A)=0$ for any $N$-acyclic $X$.
Thus, for any $Y \in \cat{D}_{Ndg}(\ma)$,
\begin{align*}
\Hom_{\cat{D}_{Ndg}(\ma)}(Y, \theta^n_q \mathbb{D}~\hat{}A)
\simeq &\Hom_{\cat{K}_{Ndg}(\ma)}(Y, \theta^n_q \mathbb{D}~\hat{}A)\\
\simeq &\Hom(\h^n_{(1)}Y(A), E).
\end{align*}
For a morphism $X \to Y$ in $\cat{D}_{Ndg}(\ma)$,
$\opn{H}_{(1)}^nX(A)\to \opn{H}_{(1)}^nY(A)$ is surjective if and only if $\Hom(\opn{H}_{(1)}^nY(A), E)\to \Hom(\opn{H}_{(1)}^nX(A), E)$ is injective.
Hence (2) is proved.
\end{proof}

Let $\ma$ be an $N_qDG$ category. 
We denote by $\cat{K}_{Ndg}^{\mathbb{P}}(\ma)$ (resp. $\cat{K}_{Ndg}^{\mathbb{I}}(\ma)$) 
the smallest full triangulated subcategory of $\cat{K}_{Ndg}(\ma)$ consisting of
$\theta_q^nA\,\hat{}$ (resp. $ \theta^n_q \mathbb{D}~\hat{}A$) for all $n \in \z$ and $A \in \ma$
which is closed under taking coproducts (resp. products).

\begin{theorem}\label{projinj}
For an $N_qDG$ category $\ma$, the following hold.
\begin{enumerate}
\item
$
\cat{D}_{Ndg}(\ma) \simeq \cat{K}_{Ndg}^{\mathbb{P}}(\ma).
$
\item
$
\cat{D}_{Ndg}(\ma) \simeq \cat{K}_{Ndg}^{\mathbb{I}}(\ma).
$
\end{enumerate}
\end{theorem}

\begin{proof}
(1)
This is essentially the Brown Representability Theorem.
Since \\ $\Hom_{\cat{K}_{Ndg}(\ma)}(\theta_q^nA\,\hat{}, N)=0$ for any $N \in \cat{K}_{Ndg}^{\phi}(\ma)$,
 we have $\Hom_{\cat{K}_{Ndg}(\ma)}(P, N)=0$ and then 
 $\Hom_{\cat{K}_{Ndg}(\ma)}(P, Y) \simeq  \Hom_{\cat{D}_{Ndg}(\ma)}(P, Y)$
for any $P \in  \cat{K}_{Ndg}^{\mathbb{P}}(\ma)$, any $Y \in \cat{K}_{Ndg}(\ma)$.
According to \cite[Theorem A]{Kr0},  there is a sequence $X_0 \to X_1 \to \cdots$ in $\cat{K}_{Ndg}^{\mathbb{P}}(\ma)$
of which the homotopy colimit $\hclim X_i$ has a morphism $\hclim X_i \to X$ in $\cat{K}_{Ndg}(\ma)$ 
such that $\Hom_{\cat{K}_{Ndg}(\ma)}(P, \hclim X_i) \simeq \Hom_{\cat{K}_{Ndg}(\ma)}(P, X)$ for any 
$P \in  \cat{K}_{Ndg}^{\mathbb{P}}(\ma)$.
Then $\Hom_{\cat{K}_{Ndg}(\ma)}(P, Z) =0$ for any $P \in  \cat{K}_{Ndg}^{\mathbb{P}}(\ma)$, where $\hclim X_i \to X \to Z \to \Sigma( \hclim X_i)$ is a triangle in $\cat{K}_{Ndg}(\ma)$.
Therefore we have $Z \in \cat{K}_{Ndg}^{\phi}(\ma)$.
Since $\hclim X_i \in \cat{K}_{Ndg}^{\mathbb{P}}(\ma)$, (1) is proved.
\par\noindent
(2) By Theorem \ref{gencogen}, $\Hom_{\cat{K}_{Ndg}(\ma)}(X, I) \simeq  \Hom_{\cat{D}_{Ndg}(\ma)}(X, I)$ for any $I \in  \cat{K}_{Ndg}^{\mathbb{I}}(\ma)$, any $X \in \cat{K}_{Ndg}(\ma)$.
And $\{ \Sigma ^i \theta^j_q A\,\hat{}\; |\; A\in \ma, 1-N \leq j \leq 0, i\in \mathbb{Z} \}$ is a set of symmetric generators for $\cat{D}_{Ndg}(\ma)$. 
According to \cite[Theorem B]{Kr0}, we similarly have the statement.
\end{proof}

By the above proof, we can check that
$(\cat{K}_{Ndg}^{\mathbb{P}}(\ma), \cat{K}_{Ndg}^{\phi}(\ma))$ and
$(\cat{K}_{Ndg}^{\phi}(\ma), \\ \cat{K}_{Ndg}^{\mathbb{I}}(\ma))$ are stable $t$-structures in $\cat{K}_{Ndg}(\ma)$
(see \cite{IKM3}).
Hence the canonical quotient $Q:\cat{K}_{Ndg}(\ma) \to \cat{D}_{Ndg}(\ma)$ has
the left adjoint
$
{\bf p} : \cat{D}_{Ndg}(\ma) \to \cat{K}_{Ndg}(\ma)
$
and a right adjoint
$
{\bf i} : \cat{D}_{Ndg}(\ma) \to \cat{K}_{Ndg}(\ma).
$
Therefore we have a recollement
\[\xymatrix{
\cat{K}_{Ndg}^{\phi}(\ma) 
\ar@<-1ex>[r]^{i_{*}}
& \cat{K}_{Ndg}(\ma)
\ar@/^1.5pc/[l]_{i^{!}} \ar@/_1.5pc/[l]^{i^{*}}  \ar@<-1ex>[r]^{Q} 
& \cat{D}_{Ndg}(\ma)
\ar@/_1.5pc/[l]^{\mathbf{P}} \ar@/^1.5pc/[l]_{\mathbf{i}}
}\]
where $i_{*}$ is the canonical embedding (see \cite{BBD},\cite{Mi1}).
\par\noindent
Let $\ma$ and $\mb$ be small $N_qDG$ categories. 
For any $NDG$ $\mb$-$\ma$-bimodule $M$, we define the functor
$-\Lqten_{\mb}M : \cat{D} _{Ndg}(\mb) \to \cat{D} _{Ndg}(\ma)$
by 
\[
X\Lqten_{\mb}M={\bf p}X\qten_{\mb}M
\]
for any right $NDG$ $\mb$-module $X$ and
the functor $\mathbb{R}\Hom_{\ma}^{\bullet q}(M, -) : \cat{D} _{Ndg}(\ma) \to \cat{D} _{Ndg}(\mb)$
by 
\[
\mathbb{R}\Hom_{\ma}^{\bullet q}(M, Y)=\Hom_{\ma}^{\bullet q}(M, {\bf i}Y)
\]
for any right $NDG$ $\ma$-module $Y$. 

\begin{definition}\label{NDGbifun}
Consider a bifunctor $\qhom_{\ma}(-,-): \cat{C}_{Ndg}(\ma)^{op} \times \cat{C}_{Ndg}(\ma) \to \cat{C}_{Ndg}(k)$.
Let $X \in \cat{C}_{Ndg}(\ma)$, $N$ a projective object of $\cat{C}_{Ndg}(\ma)$.
By Corollary \ref{exactfunc}, $\qhom_{\ma}(X,N)$ is a projective object of $\cat{C}_{Ndg}(k)$.
In the case that $k$ is a field, by Proposition \ref{NDGk}, $\qhom_{\ma}(N,X)$ is a projective object of $\cat{C}_{Ndg}(k)$
because $\qhom_{\ma}(N,X)$ is an  $NDG$ $k$-module of which all homologies are null by Proposition \ref{HhomA}.
Therefore the above functor induces a triangle bi-functor
\[
\qhom_{\ma}(-,-): \cat{K}_{Ndg}(\ma)^{op} \times \cat{K}_{Ndg}(\ma) \to \cat{K}_{Ndg}(k)
\]
and hence we have the derived functor
\[
\mathbb{R}\qhom_{\ma}(-,-): \cat{D}_{Ndg}(\ma)^{op} \times \cat{D}_{Ndg}(\ma) \to \cat{D}_{Ndg}(k).
\]
\end{definition}

\begin{remark}
If $k$ is not a field, then $\qhom_{\ma}(N,X)$ may be not projective in $\cat{C}_{Ndg}(k)$ in case of $N>2$.
\end{remark}

\begin{theorem}\label{Morita}
Let $\ma$ and $\mb$ be small $N_qDG$ categories. 
The following hold for an $NDG$ $\mb$-$\ma$-bimodule $M$.
\begin{enumerate}
\item The functor $-\Lqten_{\mb}M : \cat{D} _{Ndg}(\mb) \to \cat{D} _{Ndg}(\ma)$
is the left adjoint of the functor $\mathbb{R}\Hom_{\ma}^{\bullet q}(M, -)$.
\item The following are equivalent.
\begin{enumerate}
\item $-\Lqten_{\mb}M$ is a triangle equivalence.
\item $\mathbb{R}\Hom_{\ma}^{\bullet q}(M, -)$ is a triangle equivalence.
\item $\{ \Sigma ^i \theta^{-n}_q B\,\hat{}\otimes _{\mb}^{\bullet q}M\; |\; B\in \mb, 0 \leq n \leq N-1, i \in \z \}$
is a set of compact generators for $\cat{D} _{Ndg}(\ma)$, and
the canonical morphism 
\[
\Hom_{\cat{D}_{Ndg}(\mb)}(\theta^{j}B_1\hat{}, \Sigma^{i} B_2\hat{}) \iso 
\Hom_{\cat{D}_{Ndg}(\ma)}(\theta^{j}B_1\hat{} \qten_{\mb}M, \Sigma^{i} B_2\hat{} \qten_{\mb}M)
\]
is an isomorphism for $i=0,1$, any $j$ and any $B_1, B_2 \in \mb$.
\end{enumerate}
If $k$ is a field, then the above  are equivalent to
\begin{enumerate}
\item[(d)] $\{ \Sigma ^i \theta^{-n}_q B\,\hat{}\otimes _{\mb}^{\bullet q}M\; |\; B\in \mb, 0 \leq n \leq N-1, i \in \z \}$
is a set of compact generators for $\cat{D} _{Ndg}(\ma)$, 
the canonical map 
\[
\mb(B_1, B_2) \to 
\mathbb{R}\Hom_{\ma}^{\bullet q}(B_1\hat{} \qten_{\mb}M, B_2\hat{} \qten_{\mb}M)
\]
is an isomorphism in $\cat{D}_{Ndg}(k)$ for any $B_1, B_2 \in \mb$.
\end{enumerate}
\end{enumerate}
\end{theorem}

\begin{proof}
(a) $\Leftrightarrow$ (b)
For any right $NDG$ $\mb$-module $X$ and right $NDG$ $\ma$-module $Y$,
\begin{align*}
\Hom_{\cat{D}_{Ndg}(\ma)}(X\Lqten_{\mb}M, Y)
\simeq &\Hom_{\cat{D}_{Ndg}(\ma)}({\bf p}X\qten_{\mb}M, Y)\\
\simeq &\Hom_{K_{Ndg}(\ma)}({\bf p}X\qten_{\mb}M, {\bf i}Y)\\
\simeq &\Hom_{K_{Ndg}(\mb)}({\bf p}X, \Hom_{\ma}^{\bullet q}(M, {\bf i}Y))\\
\simeq &\Hom_{\cat{D}_{Ndg}(\mb)}(X, \Hom_{\ma}^{\bullet q}(M, {\bf i}Y))\\
\simeq &\Hom_{\cat{D}_{Ndg}(\mb)}(X, \mathbb{R}\Hom_{\ma}^{\bullet q}(M, Y)).
\end{align*}
Thus $-\Lqten_{\mb}M$ is the left adjoint of $\mathbb{R}\Hom_{\ma}^{\bullet q}(M, -)$, so that
$-\Lqten_{\mb}M$ is a triangle equivalence  if and only if so is $\mathbb{R}\Hom_{\ma}^{\bullet q}(M, -)$.
\par\noindent
(a) $\Leftrightarrow$ (c)
Note that $\Sigma ^i \theta^{-m}_q B\,\hat{}\Lqten_{\mb}M \simeq
\Sigma ^i \theta^{-m}_q B\,\hat{}\qten_{\mb}M$.
Since $\cat{D}_{Ndg}(\mb)$ has a set of compact generators 
$\{ \Sigma ^i \theta^{-m}_q B\,\hat{}\; |\; B\in \mb, 0 \leq m \leq N-1, i
\in \mathbb{Z} \}$, 
by Theorem \ref{trieqv}, 
$-\Lqten_{\mb}M$ is a triangle equivalence if and only if
$\{ \Sigma ^i \theta^{-n}_q B\,\hat{}\qten_{\mb}M\; |\; B\in \mb, 0 \leq n \leq N-1, i \in \z \}$
is a set of compact generators for $\cat{D}_{Ndg}(\ma)$ and
the canonical map 
{\footnotesize
\[
\Hom_{\cat{D}_{Ndg}(\mb)}
(\Sigma ^i \theta^{-m}_qB_1\hat{}, \Sigma ^j \theta^{-n}_qB_2\hat{}\;)
\to
\Hom_{\cat{D}_{Ndg}(\ma)}
(\Sigma ^i \theta^{-m}_qB_1\hat{} \qten_{\mb}M, \Sigma ^j \theta^{-n}_qB_2\hat{} \qten_{\mb}M)
\]
}
is an isomorphism  for any $i, j \in \z$, $0 \leq m, n \leq N-1$ and $B_1, B_2 \in \mb$.
By Theorem \ref{Sigmatheta}, we have the statement.
\par\noindent
(c) $\Leftrightarrow$ (d)
For any $B_1, B_2 \in \mb$, the canonical map 
\[
\alpha: \mb(B_1, B_2) \to 
\mathbb{R}\Hom_{\ma}^{\bullet q}(B_1\hat{} \qten_{\mb}M, B_2\hat{} \qten_{\mb}M)
\]
induces
\[\begin{aligned}
\h_{(1)}^n\mb(B_1, B_2) & \xarr{\opn{H}_{(1)}^{n}(\alpha)} \h_{(1)}^n\Hom_{\ma}^{\bullet q}
(B_1\hat{} \qten_{\mb}M, B_2\hat{} \qten_{\mb}M)\\
\h_{(N-1)}^n\mb(B_1, B_2) &  \xarr{\opn{H}_{(N-1)}^{n}(\alpha)} \h_{(N-1)}^n\Hom_{\ma}^{\bullet q}
(B_1\hat{} \qten_{\mb}M, B_2\hat{} \qten_{\mb}M)
\end{aligned}\]
for any $n$.
By Lemma \ref{hghexagon}, $\alpha$ is an isomorphism in $\cat{D}_{Ndg}(k)$ if and only if $\opn{H}_{(1)}^{n}(\alpha)$ and $\opn{H}_{(N-1)}^{n}(\alpha)$
are isomophisms for any $n$.
By Proposition \ref{HhomA}, this is equivalent to 
\[
\Hom_{\cat{D}_{Ndg}(\mb)}(\theta_q^{n}B_1\hat{}, \Sigma^{i} B_2\hat{}\;) \iso 
\Hom_{\cat{D}_{Ndg}(\ma)}(\theta_q^{n}B_1\hat{} \otimes_{\mb}^{\bullet q}M, \Sigma^{i} B_2\hat{} \otimes_{\mb}^{\bullet q}M)
\]
is an isomorphism for $i=0,1$ and any $n$.
\end{proof}

Finally, we show that the derived category $\cat{D}_{Ndg}(\ma)$ of right $NDG$ $\ma$-modules is triangle equivalent to derived categories of some ordinary $DG$ category (see also Remark \ref{DG}).

\begin{theorem}\label{NdgDG}
For any $N_qDG$ category $\ma$, 
there exists a $DG$ category $\mb$ such that
$\cat{D} _{Ndg}(\ma)$ is triangle equivalent to the derived category $\cat{D} _{dg}(\mb)$.
\end{theorem}

\begin{proof}
By Theorem \ref{Kalgtricat}, $\cat{K}_{Ndg}(\ma)$ is an algebraic triangulated category.
By Theorems \ref{projinj}, \ref{gencogen},
$\cat{D}_{Ndg}(\ma)$ is a compactly generated algebraic triangulated category
with a set $\{ \Sigma ^i \theta^j_q A\,\hat{}\; |\; A\in \ma, 1-N \leq j \leq 0, i \in \mathbb{Z} \}$ of compact generators (see Appendix for definition) .
According to \cite[7.5 Theorem]{Kr}, there are a $DG$ category $\mb$ and a triangle equivalence  $F:\cat{D} _{Ndg}(\ma) \iso \cat{D} _{dg}(\mb)$.
\end{proof}

\begin{remark}\label{DG}
We denote by $F: \cat{C}_{Ndg}(\ma)\to \cat{K}_{Ndg}(\ma)$ the canonical projection.
Let $\cat{C}_{Ndg}^{\mathbb{P}}(\ma)$ be the full subcategory of $\cat{C}_{Ndg}(\ma)$ consisting of complexes $X$ with $F(X)$ in $\cat{K}_{Ndg}^{\mathbb{P}}(\ma)$.
Then $\cat{C}_{Ndg}^{\mathbb{P}}(\ma)$ is a Frobenius category.
Let $\cat{C}^{\phi}(\cat{C}_{Ndg}^{\mathbb{P}}(\ma))$ be the category of exact complexes $X$ satisfying that
$\opn{B}^i(X)=\opn{Z}^i(X)$ for all $i$, and that $0 \to \opn{Z}^i(X) \to X^i \to \opn{Z}^{i+1}(X) \to 0 \in \mcal{E}_{\cat{C}_{Ndg}^{\mathbb{P}}(\ma)}$.
Let $\mb$ be the DG category of objects $X \in \cat{C}^{\phi}(\cat{C}_{Ndg}^{\mathbb{P}}(\ma))$ with
$\opn{Z}^0(X)=\theta^j_q A\,\hat{}$ $(1-N \leq j \leq 0)$ and the morphism  sets 
$\Hom_{\cat{C}^{\phi}(\cat{C}_{Ndg}^{\mathbb{P}}(\ma))}^{\bullet}(X, Y)$ for $X, Y \in \mb$.
According to \cite[Theorem 7.5]{Kr}, this category $\mb$ satisfies the above theorem.
\end{remark}

\section{Appendix}

In this section, we give the following results concerning Frobenius categories and triangulated categories.
Let $\mcal{C}$ be an exact category with a collection $\mcal{E}$ of exact sequences in the sense of Quillen \cite{Qu}.
An exact sequence $0 \to X \xarr{f} Y \xarr{g} Z \to 0$ in $\mcal{E}$ is called a conflation, and 
$f$ and $g$ are called an inflation and a deflation, respectively.
An additive functor $F: \mcal{C} \to \mcal{C}'$ is called exact if it sends conflations in $\mcal{C}$
to conflations in $\mcal{C}'$.
An exact category $\mcal{C}$ is called a Frobenius category provided that it has enough projectives
and enough injectives, and that any object of $\mcal{C}$ is projective if and only if it is injective.
In this case, the stable category $\underline{\mcal{C}}$ of $\mcal{C}$ by projective objects is a triangulated category
(see \cite{H1}). This stable category is called an algebraic triangulated category.

\begin{proposition}[\cite{IKM4} Proposition 7.3]\label{extr01}
Let $(\mcal{C}, \mcal{E}_{\mcal{C}}), (\mcal{C}, \mcal{E}_{\mcal{C}'})$ be Frobenius categories, 
$F: \mcal{C} \to \mcal{C}'$ an exact functor.
If $F$ sends projective objects of $\mcal{C}$ to projective objects of $\mcal{C}'$, then
it induces the triangle functor $\underline{F} : \underline{\mcal{C}} \to \underline{\mcal{C}}'$.
\end{proposition}

Let $\mcal{D}$ be a triangulated category with arbitrary coproducts.
An object $U$ in $\mcal{D}$ is compact if the canonical
$\coprod_{i}\Hom_{\mcal{D}}(U, X_i)\to \Hom_{\mcal{D}}(U, \coprod_{i }X_i)$
is an isomorphism for any set $\{X_i \}$ of objects in $\mcal{D}$.
A triangulated category $\mcal{D}$ is called compactly generated if there is a set $\cat{U}$ of compact objects
such that $\Hom_{\mcal{D}}(\cat{U}, X) = 0$ implies $X=0$ in $\mcal{D}$.

\begin{theorem}\label{trieqv}
Let $\mcal{D}$ be a compactly generated triangulated category with a set $\cat{U}$ of compact generators for $\mcal{D}$.
Let $F: \mcal{D} \to \mcal{D}'$ be a triangle functor between triangle categories such that $F$ preserves coproducts.
Then the following are equivalent.
\begin{enumerate}
\item $F$ is a triangle equivalence.
\item 
\begin{enumerate}
\item $F(\cat{U})$ is a set of compact generators for $\mcal{D}'$.
\item $\mcal{D}(U, V) \to \mcal{D}'(FU,FV)$ is an isomorphism for any $U, V \in \cat{U}$.
\end{enumerate}
\end{enumerate}
\end{theorem}

\begin{proof}
(1) $\Rightarrow$ (2) 
Trivial.
\par\noindent
(2) $\Rightarrow$ (1)
We may assume that $\cat{U}=\Sigma\cat{U}$.
Let $\mcal{U}$ be the full subcategory of coproducts of objects $U$ in $\cat{U}$.
Let $<\mcal{U}>_0:=\mcal{U}$. For $n \geq 1$ let $<\mcal{U}>_n$ be the full subcategory of $\mcal{D}$ consisting of objects $X$ such that there is a triangle $U \to V \to X \to \Sigma U$ with $U \in <\mcal{U}>_0$, $V \in <\mcal{U}>_{n-1}$.
Given $U \in \cat{U}$,
for any $U_i \in \cat{U}$ $(i \in I)$,
we have a commutative diagram
{\small
\[\xymatrix{
\Hom_{\mcal{D}}(U, \coprod_{i\in I}U_i) \ar[dd]^{\wr}\ar[r]
&\Hom_{\mcal{D}'}(FU, F\coprod_{i\in I}U_i) \ar[d]^{\wr}\\
&\Hom_{\mcal{D}'}(FU, \coprod_{i\in I} FU_i) \ar[d]^{\wr}\\
\coprod_{i\in I}\Hom_{\mcal{D}}(U, U_i)\ar[r]
&\coprod_{i\in I}\Hom_{\mcal{D}'}(FU, FU_i)
}\]
}
such that all vertical arrows are isomorphisms.
Therefore, for any $V \in \mcal{U}$
we have the canonical morphism $\Hom_{\mcal{D}}(U, V) \iso \Hom_{\mcal{D}'}(FU, FV)$ is an isomorphism.
For any $W\in <\mcal{U}>_n$ we have a triangle $U' \to V \to W \to \Sigma U$ with $U' \in <\mcal{U}>_0$, $V \in <\mcal{U}>_{n-1}$.
Applying $\Hom_{\mcal{D}}(U, -)$ to this triangle, by the induction five lemma implies that the canonical morphism $\Hom_{\mcal{D}}(U, W) \iso \Hom_{\mcal{D}'}(FU, FW)$ is an isomorphism.
According to Brown's representability theorem (e.g. \cite{Ne}), for any $Y \in \mcal{D}$, there is a sequence $Y_0 \to Y_1 \to \cdots$ with $Y_i \in <\mcal{U}>_i$ which has a triangle
\[
\coprod_{i}Y_i \xarr{1-\text{shift}} \coprod_{i}Y_i \to Y \to \Sigma \coprod_{i}Y_i
\]
Applying $\Hom_{\mcal{D}}(U, -)$ to this triangle, five lemma implies that the canonical morphism $\Hom_{\mcal{D}}(U, Y) \iso \Hom_{\mcal{D}'}(FU, FY)$ is an isomorphism.
Similarly, by the induction on $n$ for any $W \in <\mcal{U}>_n$, any $Y \in \mcal{D}$
the canonical morphism $\Hom_{\mcal{D}}(W, Y) \iso \Hom_{\mcal{D}'}(FW, FY)$ is an isomorphism.
Brown's representability theorem implies
that the canonical morphism $\Hom_{\mcal{D}}(X, Y) \iso \Hom_{\mcal{D}'}(FX, FY)$ is an isomorphism
for any $X. Y \in \mcal{D}$.
Then $F$ is fully faithful.
Therefore we have $<F(\mcal{U})>_i= F(<\mcal{U}>_i)$ for any $i$.
Since $F(\cat{U})$ is a set of compact generator for $\mcal{D}'$, for any $Z \in \mcal{D}'$,
there is a sequence $Z_0 \xarr{\beta_0} Z_1 \xarr{\beta_1} \cdots$ with $Z_i \in <F(\mcal{U})>_i$ which has a triangle
\[
\coprod_{i}Z_i \xarr{1-\text{shift}} \coprod_{i}Z_i \to Z \to \Sigma \coprod_{i}Z_i
\]
Then there is a morphism $\alpha_i:X_i \to X_{i+1}$ such that $F\alpha_i=\beta_i$ for any $i$.
Let $X$ be the homotopy colimit of the sequence $X_0 \xarr{\alpha_0} X_1 \xarr{\alpha_1} \cdots$, then
we have a commutative diagram
\[\xymatrix{
\coprod_{i}FX_i \ar[d]^{\wr}\ar[r]^{1-\text{shift}}
&\coprod_{i}FX_i \ar[d]^{\wr}\ar[r] & FX \ar@{-->}[d]^{\wr}\ar[r] & \Sigma \coprod_{i}FX_i \ar[d]^{\wr}\\
\coprod_{i}Z_i \ar[r]^{1-\text{shift}} 
&\coprod_{i}Z_i \ar[r] & Z \ar[r] & \Sigma \coprod_{i}Z_i
}\]
where first, second, fourth vertical arrows are isomorphisms.
Then third vertical arrow is an isomorphism.
Therefore $F$ is dense, and hence a triangle equivalence.
\end{proof}



\end{document}